\documentclass[reqno, 11pt]{amsart}
\pdfoutput=1
\usepackage{graphicx}
\usepackage{amsmath,amssymb,graphics,epsfig,color}
\usepackage[mathscr]{euscript}
\usepackage[all]{xy}
\xyoption{all}
\usepackage {a4wide}
\usepackage{float}
\usepackage{hyperref,color}
\theoremstyle{plain}
\newtheorem{theorem}{Theorem}[section]
\newtheorem{proposition}[theorem]{Proposition}
\newtheorem{lemma}[theorem]{Lemma}

\theoremstyle{definition}
\newtheorem{definition}[theorem]{Definition}

\theoremstyle{remark}
\newtheorem{remark}[theorem]{Remark}

\numberwithin{equation}{section}
\numberwithin{theorem}{section}

\renewcommand{\epsilon}{\varepsilon}
\renewcommand{\tilde}{\widetilde}
\renewcommand{\hat}{\widehat}

\definecolor{light}{gray}{.9}

\title[Exclusion with vorticity]{Hydrodynamic limit of an exclusion process with vorticity}
\author[L.\ De Carlo]{Leonardo De Carlo}
\address{Leonardo De Carlo \hfill\break \indent
	Center for Mathematical Analysis,  Geometry and Dynamical Systems, Instituto Superior T\'ecnico, Universidade de Lisboa, Av. Rovisco Pais, 1049-001 Lisboa, Portugal.}
\email{neoleodeo@gmail.com}

\author[D.\ Gabrielli]{Davide Gabrielli}
\address{Davide Gabrielli \hfill\break \indent
	DISIM, Universit\`a dell'Aquila
	\hfill\break\indent
	Via Vetoio,   67100 Coppito, L'Aquila, Italy}
\email{davide.gabrielli@univaq.it}

\author[P.\ Gon\c{c}alves]{Patr\'icia Gon\c{c}alves}
\address{Patr\'icia Gon\c calves\hfill\break \indent Center for Mathematical Analysis,  Geometry and Dynamical Systems, Instituto Superior T\'ecnico, Universidade de Lisboa, Av. Rovisco Pais, 1049-001 Lisboa, Portugal.}
\curraddr{}
\email{pgoncalves@tecnico.ulisboa.pt}
\thanks{}

\begin{document}

\begin{abstract}
		We construct a non reversible exclusion process with Bernoulli product invariant measure and having, in the diffusive hydrodynamic scaling, a non symmetric diffusion matrix, that can be explicitly computed. The antisymmetric part does not affect the evolution of the density but it is relevant for the evolution of the current. Switching on a weak external field we obtain a symmetric mobility matrix that is related just to the symmetric part of the diffusion matrix by an Einstein relation. We argue that this fact is typical within a class of generalized gradient models. We consider for simplicity the model in dimension $d=2$, but a similar behavior can be also obtained in higher dimensions.
		
		\bigskip
		
		
		\smallskip
		
\end{abstract}

	\maketitle
	\thispagestyle{empty}

\section{Introduction}

A major advance of the last decades in probability is the derivation of the collective behavior of models of stochastic lattice gases \cite{KL,Spohn}. In the case of diffusive models, the continuum hydrodynamic limit is given by a nonlinear diffusion equation of the form
\begin{equation}\label{idr}
\partial_t\rho=\nabla\cdot\left(\mathcal D(\rho)\nabla \rho\right)\,,
\end{equation}
where the positive definite matrix $\mathcal D$ is the diffusion matrix.
When the stochastic models satisfy a suitable constraint, called \emph{gradient condition} (see for example \cite{KL,Spohn,FHU, GK, N, S, Wick}), it is possible to derive an explicit form for the diffusion matrix. In the general reversible case $\mathcal D$ is obtained by the Green-Kubo formulas or variational representations \cite{KLS,KL,Spohn,VY}.

If we denote by $D$ the symmetric part of the diffusion matrix $\mathcal D$ we have that equation \eqref{idr}
is equivalent to the equation
\begin{equation}\label{idr-eqiv}
\partial_t\rho=\nabla\cdot\left(D(\rho)\nabla \rho\right)
\end{equation}
and $D$ is of course symmetric and positive definite.  The symmetry of $\mathcal D$ is therefore not relevant as far as the hydrodynamic equation for the evolution of the density is considered. All the models of particle systems for which a diffusive hydrodynamic limit has been proved are such that $\mathcal D$ is symmetric and therefore $\mathcal D=D$.

Equation \eqref{idr} can be rewritten as a conservation law $$\partial_t\rho+\nabla\cdot J(\rho)=0,$$
where $J(\rho)$ is the typical current associated to the density profile $\rho$. For reversible and gradient models
 \cite{MFT,BDGJL1,BDGJL2} we have that the diffusion matrix $\mathcal D$ is symmetric and the typical current is given by
\begin{equation}\label{tygr}
J(\rho)=-\mathcal D(\rho)\nabla \rho\,,
\end{equation}
which  is indeed the classical form of the Fick's law. We construct a class of generalized gradient models for wich the Fick's law \eqref{tygr} holds with a non symmetric diffusion matrix $\mathcal D$ that can be explicitly computed. We consider, for simplicity, the two dimensional setting,  but a similar result can be obtained in any dimension. This is the first example for which a non symmetric diffusion matrix is rigorously derived starting form particle systems.

For our models \eqref{tygr} can be written as
\begin{equation}\label{tygengr}
J(\rho)=-D(\rho)\nabla \rho-A(\rho)\nabla \rho,
\end{equation}
where $D$ and $A$ are respectively the symmetric and antisymmetric part of the diffusion matrix $\mathcal D$. Note that the second term of the right-hand side in \eqref{tygengr} does not
contribute to the hydrodynamic equation for the density, since it is always a divergence free term.

Introducing the orthogonal gradient defined by $\nabla^\perp F=(-\partial_y F, \partial_x F)$, the divergence free term in \eqref{tygengr} can be written as $-\nabla^\perp a(\rho)$ for a suitable function of the local density  $a(\rho)$. In this article, we discuss  examples where the matrices $D$ and $A$ can be explicitly computed.

We  point out that we have non reversible stochastic lattice gases  with an explicit product invariant measure and having a diffusive behavior. Our class of models satisfy in addition a generalized gradient condition. In general, and specially in dimension higher than one, it is difficult to construct models satisfying gradient conditions and for which it is known the invariant measure (see Section 2.4 in part II of \cite{Spohn}). The construction of the models is therefore an important part of the paper. Some generalizations are possible, but even small modifications of the rates would lead to the loss of one of the two properties.

We point out in addition that, since our models are not reversible, it is not possible to compute the diffusion matrix using the Green-Kubo formulas whose derivation uses reversibility as a basic ingredient. Applying directly the Green-Kubo formulas (see for example \cite{Spohn} section 2.2 of part II) you would get a diagonal diffusion matrix that is the wrong result.

For our special class of models in dimension 2 we have
$$
J(\rho)=-\nabla \rho-\nabla^\perp a(\rho),
$$
where $a(\rho)=2\alpha\left[\rho(1-\rho)\right]^2$ with $\alpha$ a
real parameter such that $|\alpha|<1$.

\smallskip

Switching on a smooth weak external field $E$ we obtain that the typical current becomes
\begin{equation}
J^E(\rho)=J(\rho)+\sigma(\rho)E\,,
\end{equation}
where $\sigma$ is a  positive definite matrix that is called mobility or conductivity (see \cite{Spohn} section 2.5 of part II). For gradient and reversible models we have that the mobility matrix $\sigma$ and the diffusion matrix $\mathcal D$ are related by a proportionality relation that is the Einstein relation (see again \cite{Spohn} formula 2.72 of part II). For our models we have instead that $\sigma$ is symmetric and is proportional to $D$ (i.e. just to the symmetric part of the diffusion matrix) by the Einstein relation
\begin{equation}\label{Einr}
D(\rho)=\sigma(\rho)f''(\rho)\,,
\end{equation}
where $f$ is the density of the free energy. For our solvable models we have $$f(\rho)=\rho\log (\rho)+(1-\rho)\log(1-\rho)$$ and $D(\rho)=\mathbb I$, $\sigma(\rho)=\rho(1-\rho)\mathbb I$, where $\mathbb I$ is the identity matrix.

\smallskip

The paper is organized as follows.

In Section \ref{sec:model} we introduce the stochastic lattice gases for which we prove the scaling limit. We underline the basic special features: it is a non reversible model with invariant Bernoulli product measure; satisfies a generalized gradient condition with an exact orthogonal splitting in terms of explicit local functions.

In Section \ref{SL} we introduce the empirical measure and the integrated empirical current, discuss the topological setting of the latter and state the main theorems of the paper concerning the scaling limit of both the empirical density and current.

In Section \ref{PC} we prove the convergence of the integrated empirical current.

In Section \ref{norigor} we discuss, informally, the hydrodynamic behavior of a class of generalized gradient models, the hydrodynamics in presence of a weak external field and the Einstein relation that relates just the symmetric part of the diffusion matrix and the mobility. All this is discussed without the mathematical details in order to give a general overview on the behavior of the class of models that we are considering.

\smallskip

We collect here, for a convenient consultation, the basic notation.

\subsection{Notation}\label{BN}

Consider the rescaled  two dimensional discrete torus $\Lambda_N:=\left(\mathbb Z^2/N\mathbb Z^2\right)\frac{1}{N}$, having mesh $1/N$. We denote by $E_N$  the directed edges corresponding to ordered pairs $(x,y)$ of nearest neighbor vertices of $\Lambda_N$. We call $\mathcal E_N$ the corresponding undirected ones. A generic element of $\mathcal E_N$ is written as $\{x,y\}$ when $(x,y)\in E_N$.  Note that if $(x,y)\in E_N$ then also $(y,x)\in E_N$. If $e=(x,y)\in E_N$ we denote
$e^-:=x$ and $e^+:=y$. Moreover, we call $\mathfrak e:=\{x,y\}$ the corresponding unoriented edge.
We denote by $e^{(1)}:=(1/N,0)\in \mathbb R^2$ and $e^{(2)}:=(0,1/N)\in \mathbb R^2$ the vectors of size $1/N$ and parallel to the coordinate axis.

A cycle is a sequence $\left(x^1,x^2,\dots x^n, x^{n+1}\right)$ of distinct elements of $\Lambda_N$  such that $\{x^i,x^{i+1}\}\in \mathcal E_N$ and $x^{n+1}=x^1$. We identify cycles that can be obtained one from the other by cyclic permutations (i.e. different starting points).

The lattice is embedded into $\Lambda:=\mathbb R^2/\mathbb Z^2$, the continuous two-dimensional torus of side length 1. This embedding  determines a cellular subdivision of $\Lambda$ into squares of side length $1/N$ called \emph{faces}.
An oriented face is an elementary cycle in the graph for example of the type $(x,x+e^{(1)}, x+e^{(1)}+e^{(2)},x+e^{(2)},x)$. In this case we have an anticlockwise oriented face. This corresponds geometrically to an elementary squared face having vertices $x,x+e^{(1)},x+e^{(1)}+e^{(2)},x+e^{(2)}$ plus an orientation in the anticlockwise sense. The same elementary face can be oriented clockwise and this corresponds to the elementary cycle $(x,x+e^{(2)}, x+e^{(1)}+e^{(2)},x+e^{(1)},x)$. If $f$ is a given oriented face we denote by $-f$ the oriented face corresponding to the same geometric square but having opposite orientation.
We call $F_N$ the collection of oriented faces,  $ F^\circlearrowleft_N $ the collection of the anticlockwise oriented ones and $ F^\circlearrowright_N $ the collection of the clockwise oriented ones. We call $\mathcal F_N$ the collection of unoriented faces. An unoriented face is simply determined by a collection $\{x,x+e^{(1)},x+e^{(1)}+e^{(2)},x+e^{(2)}\}$
of vertices of an elementary face. Note that both $f$ and $-f$ correspond to the same unoriented face that we call $\mathfrak f$.
Given $f\in F_N$ and $e\in E_N$ we write $e\in f$ if going around
the face $f$ according to its orientation we go through $e$ according to its orientation. Given $e\in E_N$ there are only two elements of $F_N$ to which it belongs, one is clockwise oriented while the other one is anticlockwise oriented. We call  $f^+(e)$ the anticlockwise face such that $ e\in f^+(e) $ and $ f^-(e) $ the anticlockwise face such that $ e\in -f^-(e)$. We denote by $\mathfrak f^\pm(e)$ the corresponding un-oriented faces. Given an un-oriented face $\mathfrak f$ we denote
by $f^\circlearrowleft$ and $f^\circlearrowright$, respectively, the corresponding
anticlockwise and clockwise oriented faces.

The group of translations acts naturally on the discrete torus. We denote by $\tau_x$ the translation by the element $x\in \Lambda_N$. The translations act on configurations $\eta\in \{0,1\}^{\Lambda_N}$ by $\left[\tau_x\eta\right](z):=\eta(z-x)$ and on functions $g$ by $\left[\tau_x g\right](\eta):=g(\tau_{-x}\eta)$. For notational convenience it is useful to define $\tau_{\mathfrak f}$  for an un-oriented face $\mathfrak f$. If the vertices belonging to $\mathfrak f$ are $\{x,x+e^{(1)},x+e^{(2)},x+e^{(1)}+e^{(2)}\}$ then we define $\tau_{\mathfrak f}:=\tau_x$.

Given $\Gamma\subseteq \Lambda_N$ and a configuration $\eta$, we call $\eta_{\Gamma}$ the restriction of the configuration $\eta$ to $\Gamma$. Given two configuration of particles $\eta, \eta'$ we call $\eta_{\Gamma}\eta'_{\Gamma^c}$
the configuration of particles coinciding with $\eta$ on $\Gamma$ and with $\eta'$
on $\Gamma^c$.

\section{The models}\label{sec:model}

We consider particles satisfying an exclusion rule so that the configuration space is $\Sigma_N:=\{0,1\}^{\Lambda_N}$. A generic configuration is denoted by $\eta$ and $\eta(x)$ is the occupation number of the site $x\in \Lambda_N$, that is the number of particles present at  site $x$. The number of particles at site $x$ at time $t$ is denoted by $\eta_t(x)$, which again by the exclusion rule  can be either $0$ or $1$. Given a configuration $\eta$ and $\{x,y\}\in \mathcal E_N$ we denote by $\eta^{x,y}$ the configuration obtained exchanging the occupation numbers at $x$ and $y$ while keeping fixed the configuration at the other sites, i.e.
$$
\eta^{x,y}(z):=\left\{
\begin{array}{ll}
\eta(y) & \textrm{if}\ z=x\,,\\
\eta(x) & \textrm{if}\ z=y\,,\\
\eta(z) & \textrm{otherwise}\,.
\end{array}
\right.
$$

The generator of the dynamics is  given on $f:\{0,1\}^{\Lambda_N}\rightarrow \mathbb {R}$ by
\begin{equation}\label{generatore}
\mathcal L_N f(\eta)=\sum_{(x,y)\in  E_N}c_{x,y}(\eta)\left[f(\eta^{x,y})-f(\eta)\right]\,.
\end{equation}
We fix $c_{x,y}(\eta)$ in such a way that it is zero unless
$\eta(x)=1$ and $\eta(y)=0$. In this way $c_{x,y}(\eta)$
represents the rate at which one particle jumps from $x$ to $y$
in the configuration $\eta$.
We will mainly concentrate on a specific choice for the rates $c_{x,y}(\eta)$.  In the general discussion we will however always consider generically translational covariant lattice gases, i.e. models for which the rates satisfy
\begin{equation}\label{cova}
c_{x,y}(\eta)=c_{x+z,y+z}(\tau_z\eta)\,.
\end{equation}
This corresponds to say that particles jump according to the same stochastic mechanism on any point of the lattice.

\subsection{The jump rates}\label{sec:rates}
The main model that we will consider is determined by the following choice for the rates
\begin{equation}\label{ratej}
c_{x,y}(\eta):=\eta(x)\big(1-\eta(y)\big)+\eta(x)\left[\tau_{\mathfrak f^+(x,y)}g(\eta)-\tau_{\mathfrak f^-(x,y)}g(\eta)\right]\,,
\end{equation}
where the function $g$ is a local function depending just on the occupation numbers at the vertices of the face $\{0,e^{(1)}, e^{(1)}+e^{(2)},e^{(2)}\}$. More precisely, we have
\begin{equation}\label{defg}
g(\eta)=\left\{
\begin{array}{ll}
\alpha & \textrm{if}\ \eta(0)=\eta(e^{(1)}+e^{(2)})=1\ \textrm{and}\ \eta(e^{(1)})=\eta(e^{(2)})=0\,,\\
\alpha & \textrm{if}\ \eta(e^{(1)})=\eta(e^{(2)})=1\ \textrm{and}\ \eta(0)=\eta(e^{(1)}+e^{(2)})=0\,,\\
0 & \textrm{otherwise}\,,
\end{array}
\right.
\end{equation}
where $\alpha$ is a real parameter such that $|\alpha|<1$.
This is a special case of the class of models introduced in \cite{DG} and  briefly discussed in a qualitative way  in \cite{LDC}.
The informal and intuitive description of the dynamics associated to the rates \eqref{ratej} is the following. Particles perform a  simple exclusion process, but the faces containing exactly 2 particles located at sites which are not nearest neighbors  let the particles rotate anticlockwise when $\alpha>0$
and clockwise when $\alpha<0$ with a rate equal to $|\alpha|$.

\begin{figure}[H]
	
	\centering
	
	\includegraphics{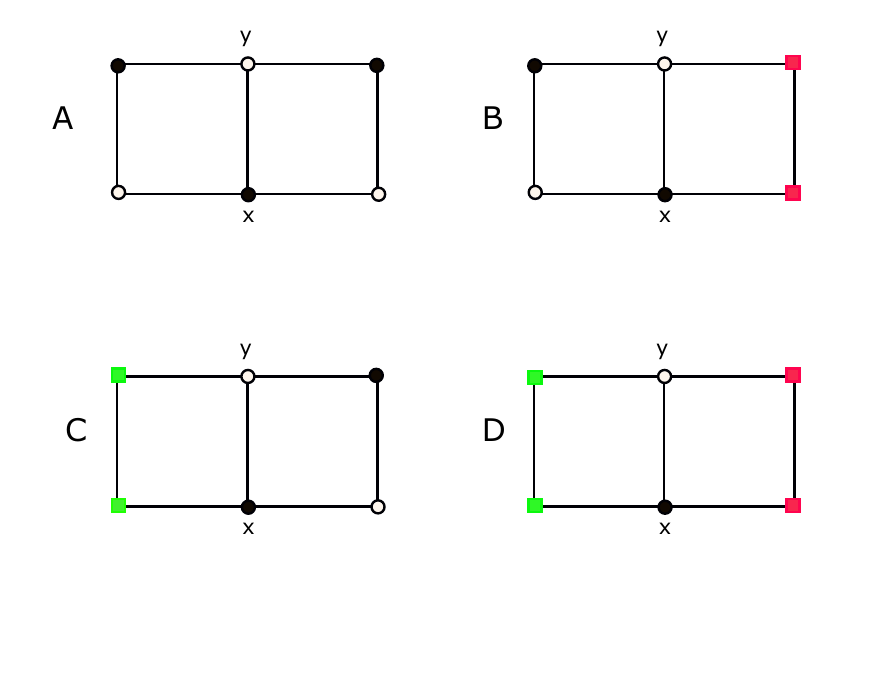}
	
	\caption{The possible configurations of particles on the two faces sharing the vertical edge $\{x,y\}$ when the lower
	vertex $x$ is occupied and the upper vertex $y$ is empty. Particles are denoted by $\bullet$, empty sites by $\circ$. Pairs of colored vertical square denote a configuration different from the one of the corresponding vertical edge in the case A. The two different colors are used to stress the fact that configurations associated to different colors may be different (case D).}
\label{ratesF}	
\end{figure}
In Figure \ref{ratesF} we illustrate the possible configurations on the two faces $\mathfrak f^\pm(x,y)$ when $y=x+e^{(2)}$ and
$\eta(x)=1-\eta(y)=1$. Particles are drawn as $\bullet$ and empty sites are drawn as $\circ$. The rate at which the particle at $x$ jumps to $y$ is given by $1$ in the configurations of type A and D, while it
is given by $1+\alpha$ in the configurations of type B, and finally it is $1-\alpha$ in the configurations of type C. In the case A we draw exactly
the structure of the configuration. In the case B, with the two red squares we indicate any configuration of particles on that vertical bond different from the one of the corresponding bond in the case A. This means that there are 3 different possible configurations of particles corresponding to the case
B. The same happens for the case C and the two green squares. In the case D, the red squares and the green squares have to satisfy the same constraints
of the previous cases and then in the case D we can have 9 different configurations of particles. Since the model is invariant by rotation, the rates for jumps on different directions (horizontal or downward) are obtained just rotating Figure
\ref{ratesF}.

We give an alternative description of the rates in Figures \ref{wei} and \ref{At}.
This is because the form of the rates is important to understand the origin of the divergence free part of the current. We fix a configuration $\eta$ and show how to determine the rates of jump across each edge $(x,y)\in E_N$. This will be zero unless $\eta(x)=1$ and $\eta(y)=0$. We show this associating some weights to the oriented edges.

In Figure \ref{wei} we show how to assign the weights. We have to search on the lattice for local configurations like the ones drawn. Black arrows correspond to weight $1$ while blue arrows correspond to weight $\alpha$ (recall that $\alpha$ can be also negative). Note that blue arrows are associated to edges along a face only in the case that the vertices of the face contain exactly $2$ particles that are at opposite corners. We call such a face \emph{activated}. We represent activated faces in Figure \ref{At} with a  \textcolor{blue}{A} drawn in the middle.
\begin{figure}[]
	\centering
	\includegraphics{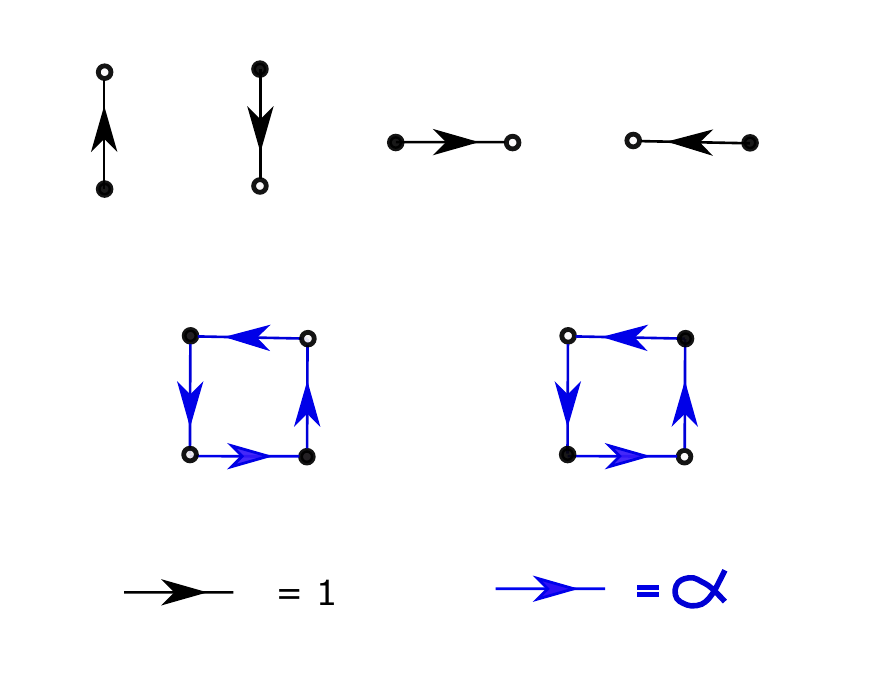}
	\caption{The graphical rules to determine the rates associated to a configuration of particles. Weights are associated to local configuration of particles as in the figure. Black arrows represent an unitary weight while blue arrows represent a weight $\alpha$. Weights concordantly oriented sum while weights oppositely oriented are subtracted.}
	\label{wei}	
\end{figure}

\begin{figure}[]	
	\centering
	\includegraphics{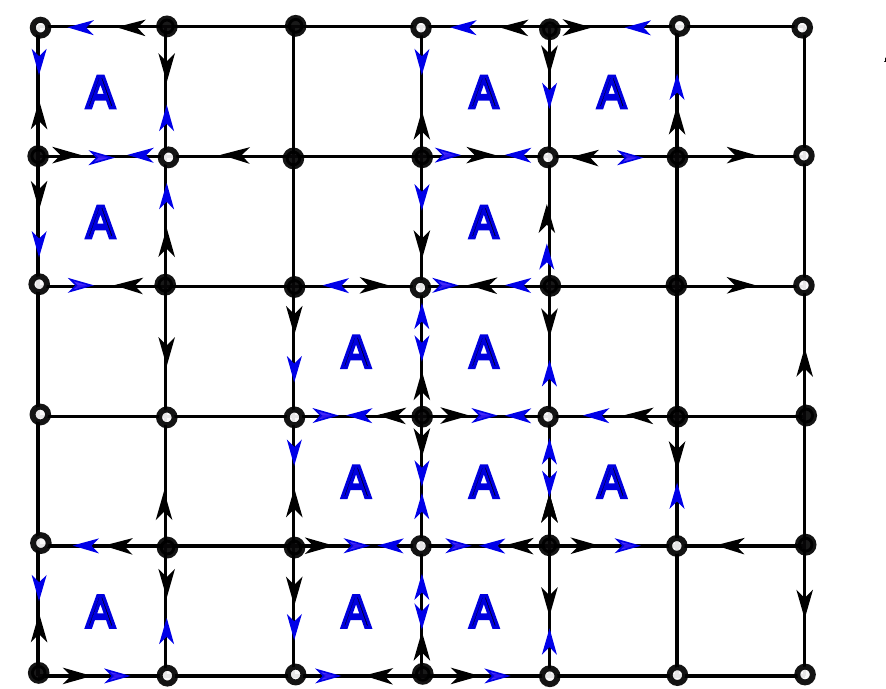}
	
\caption{An example of a configuration with the weights associated according to the local rules of Figure \ref{wei}. The \textcolor{blue}{A} inside a face means that the face is activated and weights $\alpha$ are associated to its edges counterclockwise.}
	\label{At}	
\end{figure}

We consider now the lattice with a fixed configuration of particles like in Figure \ref{At}. We search for all local configurations like in Figure \ref{wei} and assign the corresponding weights. The final weight associated to an edge is obtained summing all the weights that have been given in the procedure. Weights having the same orientation sum, while those having opposite orientations subtract.

By construction, an edge that has a non zero weight must contain a black arrow. For such an edge, we sum to the value $1$, corresponding to the black arrow, the value $\alpha$ if there is a blue arrow concordantly oriented, and $-\alpha$ if there is a blue arrow oppositely oriented. On each edge, it is possible to have just one blue arrow or two oppositely oriented. This means that on each edge with nonnegative weight there is a preferred orientation, let us say $(x,y)$, that is the one determined by the black arrow. The final weights of the graph (that correspond to the transition rates)
are obtained by giving to $(x,y)$ the weight obtained by the algebraic rules stated above (and that give always a positive result) and giving instead weight $0$ to $(y,x)$.

\subsection{Invariant measures}

We denote by $\nu_\rho$ the Bernoulli product measures on $\Sigma_N$ of parameter $\rho\in[0,1]$. We have the following.

\begin{lemma}
The Bernoulli product measures $\nu_\rho$ are invariant but not reversible (unless $\alpha=0$) for the class of models described in Section \ref{sec:rates}.
\end{lemma}
\begin{proof}
Since the dynamics is conservative it is enough to verify  that the canonical uniform measures are invariant. This corresponds to show that
\begin{equation}\label{eq:equivinv}
\sum_{(x,y)\in E_N}c_{x,y}(\eta)=\sum_{(x,y)\in E_N}c_{y,x}(\eta^{x,y})\,,
\end{equation}
for any configuration $\eta$.
The first term on the right-hand side of \eqref{ratej} corresponds to the rates of the simple exclusion process so that it satisfies this relationship. We can just restrict to the second term on the right-hand side of \eqref{ratej}. We need to check therefore that
\begin{align}
\sum_{(x,y)\in E_N}\eta(x)\left[\tau_{\mathfrak f^+(x,y)}g(\eta)-\tau_{\mathfrak f^-(x,y)}g(\eta)\right]
=\sum_{(x,y)\in E_N}\eta^{x,y}(y)\left[\tau_{\mathfrak f^+(y,x)}g(\eta^{x,y})-\tau_{\mathfrak f^-(y,x)}g(\eta^{x,y})\right]\,.\label{fognini}
\end{align}
Using the fact that $\eta^{x,y}(y)=\eta(x)$ and $\mathfrak f^\pm(y,x)=\mathfrak f^\mp(x,y)$, the right-hand side of the previous display  becomes
$$
\sum_{(x,y)\in E_N}\eta(x)\left[\tau_{\mathfrak f^-(x,y)}g(\eta^{x,y})-\tau_{\mathfrak f^+(x,y)}g(\eta^{x,y})\right]\,.
$$
Inserting this expression in \eqref{fognini},  the stationary condition becomes
\begin{align}
\sum_{(x,y)\in E_N}\eta(x)\left[\tau_{\mathfrak f^+(x,y)}g(\eta)+\tau_{\mathfrak f^+(x,y)}g(\eta^{x,y})\right]=\sum_{(x,y)\in E_N}\eta(x)\left[\tau_{\mathfrak f^-(x,y)}g(\eta)+\tau_{\mathfrak f^-(x,y)}g(\eta^{x,y})\right]\,.\label{eq:group}
\end{align}
Let us fix an un-oriented  face $ \mathfrak{f}$ and let us call $f^\circlearrowleft\in F^{\circlearrowleft}_N$ the corresponding  anticlockwise oriented face. The contribution from the left-hand side of \eqref{eq:group} that contains functions shifted by $\tau_{\mathfrak f}$ is given by
\begin{equation}\label{acea}
\sum_{(x,y)\in f^\circlearrowleft }\eta(x)\left[\tau_{\mathfrak f}g(\eta)+\tau_{\mathfrak f}g(\eta^{x,y})\right]\,,
\end{equation}
while instead from the right-hand side of \eqref{eq:group} we have
\begin{equation}\label{acea2}
\sum_{(x,y)\in f^\circlearrowright }\eta(x)\left[\tau_{\mathfrak f}g(\eta)+\tau_{\mathfrak f}g(\eta^{x,y})\right]\,.
\end{equation}
We claim that \eqref{acea} and \eqref{acea2} coincide for any configuration $\eta$ and for any face $\mathfrak f$. Note that in both expressions  \eqref{acea} and  \eqref{acea2}, we have local functions depending just on the occupation numbers on the face $\mathfrak f$. We need then just to cheek the validity of this relationship for any possible configuration of particles on the face $\mathfrak f$, disregarding the configuration's values  outside the face.

When the total number of the particles on the face $\mathfrak f$ is equal to $0,1,3,4$ then the equality between \eqref{acea} and \eqref{acea2} is immediate, since all the terms are identically zero.
The only non trivial case is when the total number of particles  equals $2$. In Figure \ref{fig: dueparticelle} we check the validity of this statement for two special configurations $\eta$. All the remaining cases are obtained from these by a suitable rotation.
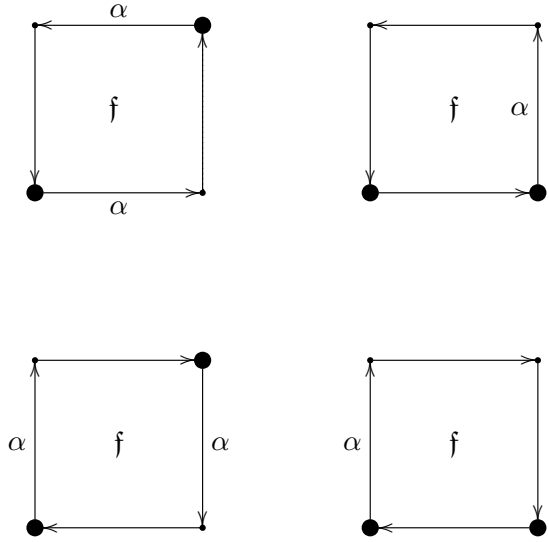
\begin{figure}[]
	
	\[
	\hspace{0.1cm}
	\entrymodifiers={+<0.5ex>[o][F*:black]}
	\xymatrix@C=2cm@R=2cm{\ar[d]
		&  *+[o][F*:black]{} \ar@{.}[d] \ar[l]_{\textit{\normalsize $ \alpha $}}& \ar[d]  & \ar[l]  \\
		*+[o][F*:black]{} \ar[r]_{\textit{\normalsize $ \alpha $}}\ar@{}[ur]|{\textit{\normalsize $\mathfrak f$ }}  &  \ar[u]   &   *+[o][F*:black]{}\ar@{}[ur]|{\textit{\normalsize $\mathfrak f$}}\ar[r]  & *+[o][F*:black]{}   \ar[u]^{\textit{\normalsize $ \alpha $}} \\
		\ar[r]
		&  *+[o][F*:black]{} \ar[d]^{\textit{\normalsize $ \alpha $}}&   \ar[r] & \ar[d] \\
		*+[o][F*:black]{} \ar[u]^{\textit{\normalsize $ \alpha $}}\ar@{}[ur]|{\textit{\normalsize $\mathfrak f$}}  &  \ar[l]   &   *+[o][F*:black]{}\ar[u]^{\textit{\normalsize $ \alpha $}}\ar@{}[ur]|{\textit{\normalsize $\mathfrak f$}}  & *+[o][F*:black]{}  \ar[l]  \\
	}
	\vspace{0.3cm}
	\]
	\caption{\small{On the top we represent the value of \eqref{acea} for two different configurations of particles (left and right respectively). Since \eqref{acea} corresponds to the sum of 4 terms associated to directed edges around the face, we write near to each edge the corresponding contribution.  If near an edge there is anything written, it means that the corresponding contribution is zero. At the bottom we have the same but for \eqref{acea2}. Note that on the top, the edges form anticlockwise oriented cycles while, at the bottom, we have clockwise oriented cycles. The values of \eqref{acea} and \eqref{acea2} for a configuration of particles is obtained summing all the values associated to the edges on the corresponding figure.}}
		\label{fig: dueparticelle}
	
\end{figure}
Therefore,  \eqref{eq:group} can then be rewritten as
\begin{align}
\sum_{\mathfrak f\in \mathcal F_N}\sum_{(x,y)\in f^\circlearrowleft}\eta(x)\left[\tau_{\mathfrak f}g(\eta)+\tau_{\mathfrak f}g(\eta^{x,y})\right]=\sum_{\mathfrak f\in \mathcal F_N}\sum_{(x,y)\in f^\circlearrowright}\eta(x)\left[\tau_{\mathfrak f}g(\eta)+\tau_{\mathfrak f}g(\eta^{x,y})\right]\,,\label{eq:group2}
\end{align}
which  is satisfied since there is equality for each $\mathfrak f$.
We recall that in \eqref{eq:group2} we are again using the convention that given $\mathfrak f\in \mathcal F_N$ we denote by $f^\circlearrowleft\in F_N^{\circlearrowleft}$  and
by $f^\circlearrowright\in F_N^{\circlearrowright}$, respectively, the corresponding anticlockwise oriented  and clockwise oriented faces.

The fact that, unless $\alpha=0$, the dynamics is not reversible, can be easily shown verifying that the detailed balance condition is not satisfied.
\end{proof}

\subsection{Generalized gradient condition}

In this section we show that the model we introduced satisfies a generalized gradient condition. We start recalling a classic discrete Hodge decomposition.

\subsubsection{Discrete vector fields and Hodge decomposition}

A discrete vector field is a map $\phi: E_N\to \mathbb R$ that satisfies the antisymmetry property $\phi(x,y)=-\phi(y,x)$. The divergence of a discrete vector field is defined by
\begin{equation}\label{divergence}
\nabla\cdot \phi(x):=\sum_{y: (x,y)\in E_N}\phi(x,y)\,.
\end{equation}
A discrete vector field is called of gradient type if there exists a function $f:\Lambda_N\to \mathbb R$ such that
$\phi(x,y)=f(y)-f(x)$ and in this case we write shortly $\phi=\nabla f$.

We recall briefly the Hodge decomposition for discrete vector fields considering  the two dimensional discrete torus $\Lambda_N$. We  denote by $\Gamma^0$ the collection of real valued functions defined on the set of vertices, that is:
$\Gamma^0:=\{g\, :\, \Lambda_N\to \mathbb R\}$. We denote by  $\Gamma^1$ the vector space of discrete vector fields
endowed with the scalar product
\begin{equation}\label{sc}
\langle\phi,\varphi\rangle:=\sum_{(x,y)\in E_N}\phi(x,y)\varphi(x,y)\,, \qquad \phi,\varphi\in \Gamma^1\,.
\end{equation}
Finally, we denote by $\Gamma^2$ the vector space of 2-forms. A 2-form is a map $\psi$ from the set of oriented faces $F_N$ to $\mathbb R$ which  is antisymmetric with respect to the change of orientation i.e. such that $\psi(-f)=-\psi(f)$. The boundary $\delta\psi$ of a 2-form $\psi$ is a discrete vector field defined by
\begin{equation}\label{defd}
\delta\psi(e):=\sum_{f\,:\, e\in f}\psi(f)\,.
\end{equation}
By construction $\nabla\cdot \delta\psi=0$ for any $\psi$.
The two-dimensional discrete Hodge decomposition \cite{B,L} is written as the direct sum
\begin{equation}\label{dish2}
\Gamma^1=\nabla \Gamma^0\oplus \delta\Gamma^2\oplus\Gamma^1_H\,,
\end{equation}
where the orthogonality is with respect to the scalar product given in  \eqref{sc}. The discrete vector fields on $\nabla\Gamma^0$ are the gradient ones. The dimension of $\nabla\Gamma^0$ is $N^2-1$. The vector subspace $\delta\Gamma^2$ contains all the discrete vector fields that can be obtained by \eqref{defd} from a given 2-form $\psi$. The dimension of $\delta\Gamma^2$ is $N^2-1$. Elements of $\delta\Gamma^2$ are called \emph{circulations}. The dimension of $\Gamma_H^1$ is simply 2. Discrete vector fields in $\Gamma^1_H$ are called \emph{harmonic}. A basis in $\Gamma^1_H$ is given by the vector fields $\varphi^{(1)}$ and $\varphi^{(2)}$ defined by
\begin{equation}\label{cenes}
\varphi^{(i)}\left(x,x+e^{(j)}\right):=\delta_{i,j}\,, \qquad i,j=1,2\,.
\end{equation}
The decomposition \eqref{dish2} holds in any dimension. For the $d-$dimensional torus the dimension of the
harmonic subspace is $d$.
Given a discrete vector field $\phi\in \Gamma^1$ we write
\begin{equation}\label{hooo}
\phi=\phi^\nabla+\phi^\delta+\phi^H
\end{equation}
to denote the unique splitting in the three orthogonal components.

\subsubsection{Generalized gradient condition}

To determine the scaling limits of a model, a key role is played by the \emph{instantaneous current} $j_\eta(x,y)$. For a generic dynamics having generator in the form  \eqref{generatore}, this is defined by
\begin{equation}\label{defist}
j_\eta(x,y):=c_{x,y}(\eta)-c_{y,x}(\eta)\,.
\end{equation}
The translational covariance property of the rates  given in \eqref{cova}
is inherited by the instantaneous current.
For any fixed configuration $\eta$, $j_\eta$ is a discrete vector field.
The classic form of the gradient condition for stochastic lattice gases requires that the instantaneous current can be written as a gradient
\begin{equation}\label{grh}
j_\eta(x,y)=\tau_yh(\eta)-\tau_xh(\eta)
\end{equation}
for a suitable local function $h$. In order to compute the hydrodynamic scaling limit of a model, it is useful to be able to perform a double discrete summation by parts (see some details in Section \ref{SL}). This double summation  by parts  is possible under some generalized gradient condition. We consider two of them  and then we show that they are indeed equivalent.

The first generalized gradient condition is, for example, the one in Definition 2.5  page 61 of \cite{KL}. For simplicity we restrict ourselves to the case of nearest neighbours jumps.
\begin{definition}\label{def1}
A stochastic lattice gas satisfies a generalized gradient condition if its instantaneous current can be written as
	\begin{equation}\label{grgen}
	j_\eta(x,x+e^{(i)})=\sum_{n=1}^{n_0}\sum_{y\in \Lambda_N}
	p_{i,n}(y-x)\tau_yh_{i,n}(\eta)\,.
	\end{equation}
	In the above formula $i$ is an index that labels the dimension so that $i=1,\dots,d$ ($d=2$ in our case) while $n_0$ is a given natural number.
	
We have a collection of local functions $h_{i,n}$
	and a collection of finite support functions $p_{i,n}$ such that
	$\sum_{z\in \Lambda_N}p_{i,n}(z)=0$ for any $i,n$.	
\end{definition}

The finite support condition is relevant when the model is defined on an infinite lattice. Since in our setting the lattice is finite, any function is local. In this framework the finite support condition means that the support is finite and does not depend on the size $N$ of the lattice. More precisely the functions considered do not depend on $N$ too.

The second possible generalized gradient condition can be stated as follows (see \cite{S}). We introduce the \emph{gradient space} $\mathcal G$ defined by the collection of functions $g:\{0,1\}^{\Lambda_N}\to \mathbb R$ that can be written as
\begin{equation}
\mathcal G:=\left\{ g\,:\, g=\sum_{i=1}^d\left(\tau_{e^{(i)}}h_i-h_i\right)\right\},
\end{equation}
where $\left\{h_i\right\}_{i=1}^d$ is a collection of local functions.
\begin{definition}\label{def2}
A stochastic lattice gas satisfies a generalized gradient condition if
\begin{equation}\label{secondgrad}
j_\eta(0,e^{(i)})\in \mathcal G\,, \qquad i=1,\dots , d\,.
\end{equation}
\end{definition}
Note that by the translational covariance property of the instantaneous current, the Definition \ref{def2} implies that there exist some local functions $\left\{h_{i,j}\right\}_{j=1}^d$ such that
\begin{equation}\label{alvar}
j_\eta(x,x+e^{(i)})=\sum_{j=1}^d\left(\tau_{x+e^{(j)}}h_{i,j}-\tau_xh_{i,j}\right)\,, \qquad \forall i=1,\dots, d\,.
\end{equation}
We will now show that Definitions \ref{def1} and \ref{def2} are indeed equivalent.
\begin{lemma}
Definitions \ref{def1} and \ref{def2} are equivalent.
\end{lemma}
\begin{proof}
First observe that  \eqref{alvar} coincides with \eqref{grgen} with $n_0=d$ for a special choice of the functions $p$'s, so that Definition \ref{def2} is a special case of Definition \ref{def1}. Conversely, we will now show  that any current given as  in \eqref{grgen} can be rewritten as in  \eqref{alvar}.

For simplicity we discuss the case $d=2$. Any signed measure $p$ on $\Lambda_N$ with finite support and having equal positive and negative mass, i.e. such that $\sum_x p(x)=0$ can be decomposed as $p=\sum_\ell \hat {p}^\ell$ where each $\hat{p}^\ell$ is a signed measure of the form $p^\ell=m_\ell\left(\delta_{x^{\ell,+}}-\delta_{x^{\ell,-}}\right)$, where $x^{\ell,\pm}$ are elements of the lattice and $m_\ell$ are positive numbers. The proof of this fact is rather elementary and corresponds to write a signed measure as a convex combination
of the extremal ones. This decomposition is not unique. Using this fact,  \eqref{grgen} is equivalent to
\begin{equation}
\label{grgenaltra}
j_\eta(x,x+e^{(i)})=\sum_k \left(\tau_{x+x^{i,k,+}}\hat{h}_{i,k}(\eta)-\tau_{x+x^{i,k,-}}\hat{h}_{i,k}(\eta)\right)
\end{equation}
where  $x^{i,k,\pm}$ are points of the lattice and the functions $\hat h$ are obtained multiplying the local functions $h$ by the coefficients $m_\ell^{i,n}$
of the extremal decomposition of $p_{i,n}$. Take now a local function $h$ and consider $\tau_yh-\tau_xh$ where $x=(x_1,x_2)$ and $y=(y_1,y_2)=(x_1+n_1e^{(1)},x_2+n_2e^{(2)} )$, $n_i>0$. Other cases can be discussed in the same way.
If we define the local functions
\begin{equation}\label{lfg}
\left\{
\begin{array}{l}
g_1=\sum_{j=0}^{n_1-1}\tau_{je^{(1)}}h\,,\\
g_2=\sum_{j=0}^{n_2-1}\tau_{n_1e^{(1)}+je^{(2)}}h\,,
\end{array}
\right.
\end{equation}
by construction we have
\begin{equation}
\nonumber
\left\{
\begin{array}{l}
\tau_{x+e^{(1)}}g_1-\tau_xg_1=\tau_{x+ne^{(1)}}h-\tau_xh\,,\\
\tau_{x+e^{(2)}}g_2-\tau_xg_2=\tau_yh-\tau_{x+ne^{(1)}}h\,.
\end{array}
\right.
\end{equation}
We obtain therefore
\begin{equation}\label{conc}
\tau_yh-\tau_xh=\left(\tau_{x+e^{(1)}}g_1-\tau_xg_1\right)+\left(\tau_{x+e^{(2)}}g_2-\tau_xg_2\right).
\end{equation}
Using  \eqref{conc},  we can construct some proper local functions $ \{g^{i,k}_j\}_{j=1,2} $ such that we can rewrite \eqref{grgenaltra} as
\begin{equation}\label{eq:grgenaltra2}\begin{split}
j_\eta(x,x+e^{(i)})&=\underset{k}{\sum} \,\underset{j=1}{\overset{2}{\sum}}\,\left(\tau_{x+x^{i,k,-}+e^{(j)}}g^{i,k}_j(\eta)-\tau_{x+x^{i,k,-}}g^{i,k}_j(\eta)\right)\\
&= \underset{j=1}{\overset{2}{\sum}}\,\left(\tau_{x+e^{(j)}}\left(\underset{k}{\sum}
\tau_{x^{i,k,-}}g^{i,k}_j(\eta)\right)-\tau_{x}\left(\underset{k}{\sum}\tau_{x^{i,k,-}}g^{i,k}_j(\eta)\right)\right).\nonumber
\end{split}\end{equation}
Defining $ h^{i,j}(\eta):=\underset{k}{\sum}\,\tau_{x^{i,k,-}}g^{i,k}_j(\eta) $, we have showed that \eqref{grgen} can be rewritten as  \eqref{alvar},  for a suitable choice of local functions.

\end{proof}
Since the two definitions are equivalent we will use the simpler one given in  \eqref{alvar}.

\subsubsection{Functional Hodge decomposition}\label{ss:func hod}
We briefly discuss some geometric features of the above generalized gradient conditions.
Consider the class of discrete vector fields $j_\eta$
that depend, in a translational covariant way on the  configurations of particles, i.e.
such that $$j_\eta(x,y)=j_{\tau_z\eta}(x+z,y+z).$$  As we already discussed, the instantaneous current for a translational covariant model of interacting particles, is always of this type. According to the results in \cite{DG} (in particular Theorems 1 and 2 there), there exists a functional version of the discrete Hodge decomposition \eqref{dish2}. Discrete vector fields of the form \eqref{grh} play the role of gradient discrete vector fields. The functional version of the circulations, in dimension 2, is given by the vector fields that can be written as
\begin{equation}\label{cirh}
j_\eta(x,y)=\tau_{\mathfrak f^+(x,y)}g(\eta)-\tau_{\mathfrak f^-(x,y)}g(\eta),
\end{equation}
for a suitable function $g$. Note that both \eqref{grh} and \eqref{cirh} are translational covariant and, moreover, for any fixed $\eta$ we have that \eqref{grh} is an element of $\nabla \Gamma^0$ while \eqref{cirh} is an element of $\delta \Gamma^2$. The role of harmonic vector fields is played by vectors of the form
\begin{equation}\label{Ch}
C_i(\eta)\varphi^{(i)}\,, \qquad  i=1,2,
\end{equation}
where $C_i$ are functions on configurations, which are invariant by translations and $\varphi^{(i)}$ are defined by \eqref{cenes}. Theorem 2 in \cite{DG} says that any translational covariant discrete vector field $j_\eta$ can be written in a unique way (up to a suitable addition of translation invariant functions) as the sum of a term of the form \eqref{grh}, a term of the form \eqref{cirh} and two terms of the form \eqref{Ch}, one for each $i$. The important fact of this decomposition is that the functions $h$ and $g$ are not necessarily local and the decomposition may depend on the size $N$ of the lattice.

Consider an instantaneous current satisfying \eqref{alvar}. Then by a direct computation it is possible to check that
$$
\sum_{x\in \Lambda_N}j_\eta(x,x+e^{(j)})=\langle j_\eta, \varphi^{(j)}\rangle=0\,, \qquad \forall \eta\,, \forall j=1,\dots d\,.
$$
This means that for each fixed $\eta$ the discrete vector field $j_\eta$
is orthogonal to the harmonic subspace.
This implies that the functions $C^{(i)}$ in formula (69) of Theorem 2 in \cite{DG}, that correspond to the ones in equation \eqref{Ch}, are identically zero. The instantaneous current for any model satisfying \eqref{alvar} can therefore be
written as
\begin{equation}\label{istcur-1}
j_\eta(x,y)= \big[\tau_yh(\eta)-\tau_x h(\eta)\big]+\big[\tau_{\mathfrak f^+(x,y)}g(\eta)-\tau_{\mathfrak f^-(x,y)}g(\eta)\big]\,,
\end{equation}
for suitable functions $h$ and $g$, not necessarily local.

The model with rates \eqref{ratej} has the peculiar feature that the instantaneous current can be decomposed like \eqref{istcur-1} with $h$ and $g$ being local functions.
Indeed by a direct computation using the special form of the local function defined in \eqref{defg} we have that for the rates in \eqref{ratej} the instantaneous current has the form \eqref{istcur-1} with $h(\eta)=-\eta(0)$ and the function $g$ corresponding to the one defined by \eqref{defg}.

 For an instantaneous current like \eqref{istcur-1} we call respectively
 \begin{equation}\label{eq:j^g&j^g}
 j^\nabla_\eta(x,y):= \big[\tau_yh(\eta)-\tau_x h(\eta)\big] \text{ and  } j^\delta_\eta(x,y):= \big[\tau_{\mathfrak f^+(x,y)}g(\eta)-\tau_{\mathfrak f^-(x,y)}g(\eta)\big]\
 \end{equation}
 the gradient part of the current $ j^\nabla_\eta(x,y) $ and the circulation part of the current $ j^\delta_\eta(x,y) $. At the end of next section we will observe that the hydrodynamics of the particle system will be related only to the gradient part, when we observe just the density. The circulation part is relevant when we observe instead the current  too.

\section{Scaling limits}\label{SL}

There are two natural empirical objects suitable to describe the scaling limit of the model: the empirical measure and the empirical integrated current. We consider the model just in dimension 2 but in some formulas we keep the notation $d$ for the dimension. We do this to make some definitions clearer. For the specific computations of this paper $d$ can always be substituted by $2$.

\subsection{Empirical measure and current}
Let $\mathcal M^+(\Lambda)$ be the space of finite positive measures on $\Lambda$ with total mass $\leq 1$ and endowed with the weak topology. Let $\pi_N: \Sigma_N\to \mathcal M^+(\Lambda)$ be the map that associates to the configuration $\eta $ its  empirical measure $\pi^N(\eta,du)$ defined by
\begin{equation}
\pi^N(\eta,du):=\frac{1}{N^d}\sum_{x\in \Lambda_N}\eta(x)\delta_{x}(du)\,,
\end{equation}
where $\delta_v(du)$ denotes a Dirac measure at $v\in \Lambda$. Let $\rho:\Lambda\rightarrow [0,1]$ be a measurable function.  We say that a sequence of probability measures $\mu_N$ on $\Sigma_N$ is associated to the density profile $\rho$ if for any $f\in C(\Lambda)$
\begin{equation}\label{associato}
\lim_{N\to +\infty}\mu_N\left(\eta\in\Sigma_N: \left|\int_{\Lambda} f(u)\,\pi^N(\eta,du)-\int_{\Lambda}f(u)\rho(u)du\right|>\epsilon\right)=0\,, \qquad \forall \epsilon>0\,.
\end{equation}
This is the same as saying that the sequence of measures $\pi^N(\eta,du)\in \mathcal M^+(\Lambda)$ converges weakly, as $N\to+\infty$, and in probability with respect to $\mu_N$,  to the measure $\rho(u)du\in \mathcal M^+(\Lambda)$.  With a small abuse of notation we denote again by $\pi^N$ the map
$\pi^N :\mathcal{D}([0,T], \Sigma_N)\to \mathcal{D}([0,T], \mathcal{M}^{+}(\Lambda) ) $ that associates to the trajectory
$(\eta_t)_{t\in[0,T]}$ the path $\pi^N\left[(\eta_t)_{t\in[0,T]}\right]:= (\pi_N(\eta_t))_{t\in[0,T]}$ and we set
$\pi^N_t(du):= \pi^N(\eta_t,du)$.
We denote by $\mathbb P_{\mu_N}$
the probability measure on $\mathcal{D}([0,T], \Sigma_N)$ when the particles are distributed at time zero according to the sequence of probability measures $\mu_N$ and the Markovian dynamics is determined by the rates \eqref{ratej}  multiplied by a factor of $N^2$,
while we indicate with $ \mathbb{Q}_N $ the probability measure induced by the empirical measure on the space of c\`adl\`ag trajectories $ \mathcal{D}([0,T],\mathcal{M}^+(\Lambda)) $ that is $ \mathbb{Q}_N:=\mathbb{P}_{\mu_N}\circ (\pi^N)^{-1}$.
 The expectation with respect to  $\mathbb P_{\mu_N}$ will be denoted by $\mathbb E_{\mu_{N}}$.

We define now the integrated empirical current field $\mathcal J^N_t(\cdot)$ as a bounded linear functional on a proper Hilbert space as it will be explained in Section \ref{sec:currentop}. This is due to technical issues related to the tightness of the related distribution measures. Scaling limits and large deviations for the integrated empirical current were proven in \cite{BDGJL3} for the simple exclusion process but the proof of the tightness is incomplete there. The topological setting that we use in this paper follows the approach of \cite{BLG}.

For any $ t\in [0,T] $ we denote, respectively, by $ \mathcal{N}^N_{x,y}(t) $ and $ \mathcal{N}^N_{y,x}(t) $ the number of particles that crossed the bond $ (x,y) $ and the ones that crossed the bond $ (y,x) $ in the time window $ [0,t]$ for the process with rates multiplied by $N^2$. The integrated current  field (in the following simply the current field) is  then introduced as the functional acting on continuous vector fields $G:\Lambda \rightarrow \mathbb R^d$ in the following way
\begin{equation}\label{orleans}
 \mathcal J^N_t(G):=\frac{1}{N^d}\sum_{(x,y)\in E_N}G_N(x,y)\mathcal N^N_{x,y}(t),
\end{equation}
where
\begin{equation}\label{discvec}
G_N(x,y):=\int_{(x,y)}G\cdot dl\,
\end{equation}
is the line integral on the oriented segment $(x,y)$ of the vector field $G$. Note that $|G_N(x,y)|\leq |G|_\infty/N$ where $|G|_\infty:=\sup_{i=1,\dots ,d}\sup_{x\in \Lambda}|G_i(x)|$.
Formula \eqref{orleans} can be written also as
\begin{equation}\label{orleans2}
\frac{1}{N^d}\sum_{\{x,y\}\in \mathcal E_N}G_N(x,y)J^N_{x,y}(t),
\end{equation}
where $J^N_{x,y}(t):=\mathcal N^N_{x,y}(t)-\mathcal N^N_{y,x}(t)$. Note that the two expressions  \eqref{orleans} and \eqref{orleans2} are equivalent  since on the right-hand side of \eqref{orleans2} there is a product of two antisymmetric functions (on the pair $x,y$) and the expression is not ambiguous.

The microscopic continuity equation related to the conservation of mass is
\begin{equation}\label{eq:microcl}
\eta_t(x)-\eta_0(x)+\nabla\cdot J^N_t(x)=0\,, \qquad \forall x\in \Lambda_N\,, t>0\,,
\end{equation}
where $\nabla \cdot J^N_t (x)$ is defined in \eqref{divergence}.
This equation allows deriving,  in a weak sense, a discrete continuity equation relating the empirical measure and the current field, namely we have that
\begin{equation}\label{cont}
\int_\Lambda f \pi^N(\eta_t,du)-\int_\Lambda f \pi^N(\eta_0, du)-\mathcal J^N_t\left(\nabla f\right) =0\,, \qquad \forall f\in C^1(\Lambda)\,, \forall t\in[0,T]\,.
\end{equation}
From the general theory of interacting particle systems (see \cite{Spohn} part II Section 2.3) we have that
\begin{equation}\label{eq:martingale}
M^N_{x,y}(t):=J^N_{x,y}(t)-N^2\int^t_0 ds \,j_{\eta_s}(x,y)\,, \qquad (x,y)\in E_N
\end{equation}
is a martingale with respect to the natural filtration and therefore $\mathbb E_{\mu_N}(M^N_{x,y}(t))=0$, for any initial condition $\mu_N$. Considering  a test vector field $G$ we obtain the martingales
\begin{equation}\label{martc}
M^N_G(t):=  \mathcal J^N_t(G)-
\frac{N^{2-d}}{2}\int_0^t\sum_{(x,y)\in E_N}j_{\eta_s}(x,y)G_N(x,y) ds.
\end{equation}
The factor $1/2$ in the above formula appears since for a symmetric function $s(x,y)$ we have
$\sum_{\{x,y\}\in \mathcal E_N}s(x,y)=1/2\sum_{(x,y)\in  E_N}s(x,y)$.
The martingale in \eqref{martc} can be transformed, in the case of the rates \eqref{ratej}, using some discrete integration by parts and the special form of the instantaneous current into
\begin{equation}\label{eq:martc2}
M^N_G(t)= \mathcal J^N_t(G) - N^{2-d}\int_0^tds\left(\sum_{x\in \Lambda_N}\eta_s(x)\nabla\cdot G_N(x) +\sum_{\mathfrak f\in \mathcal F_N}\tau_{\mathfrak f}g(\eta_s)\sum_{(x,y)\in f^\circlearrowleft}G_N(x,y)\right).
\end{equation}
On the right-hand side of the above equation the first term inside the integral corresponds to a discrete divergence, while the second one is a discrete version of a two-dimensional curl. By Taylor expansion, for a $C^3$ vector field $G$, we have indeed
\begin{equation}\label{eq:aproxnablas}
\left\{
\begin{array}{l}
N^2\nabla\cdot G_N(x)=\nabla\cdot G(x)+O(1/N^2)\,,\\
N^2\sum_{(x,y)\in f^\circlearrowleft}G_N(x,y)=\nabla^{\perp}\cdot G(z)+O(1/N^2),
\end{array}
\right.
\end{equation}
where the infinitesimal terms are uniform, $z$ is the center of the face $\mathfrak f$
and  we used the notation
$\nabla^\perp\cdot G(z):=-\partial_{z_2}G_1(z)+\partial_{z_1}G_2(z)$.

The error terms can be estimated by $(C/N^2) \underset{\left\{k, i+j \leq 3\right\}}\sup|\partial^i_{x_1}\partial^j_{x_2}G_k|_\infty$, where $C$ is a universal constant. By \eqref{eq:aproxnablas} the sums in  \eqref{eq:martc2} are directly related to discretized versions of differential operations on the vector field $G$ and they can be approximated by Riemann sums up to negligible terms.

\begin{remark}\label{re:divj^g}
	For  the currents $ j^\nabla_\eta(x,y) $ and  $ j^\delta_\eta(x,y) $ of \eqref{eq:j^g&j^g}, we have that
	 \begin{equation}\label{eq:divj^g}
	 \nabla\cdot j^\delta_\eta (x)=0,\, \forall x\in \Lambda_N \Rightarrow \nabla\cdot(j^\nabla_\eta +j^\delta_\eta)(x)= \nabla\cdot j^\nabla_\eta (x), \,\forall x\in \Lambda_N ,
	 \end{equation}
	 this means that the hydrodynamics of the empirical measure  will be related only to the gradient part of the instantaneous current,  because the continuity equation \eqref{cont} was  obtained from the  microscopic conservation law \eqref{eq:microcl} and the current $J_t^N$ is related to the instantaneous current by the martingales \eqref{eq:martingale}.
\end{remark}

\begin{remark}\label{re:qmart}
From the general theory (see for example \cite{Spohn} Section II.2.3 or \cite{KL} Appendix 1 Section 5) we define the  martingale $N^N_G(t)$
\begin{equation}\label{eq:martN}
N^N_G(t)= \left[M^N_G(t)\right]^2-N^{2-2d}\int^t_0 ds \sum_{\{x,y\}\in \mathcal{E}_N}(c_{x,y}(\eta_s)+c_{y,x}(\eta_s))G^2_N(x,y).
\end{equation}
The second term on the right-hand side of last display  is called the quadratic variation of $M^N_G$.
Since $N^N_G$ is a martingale we have
\begin{equation}\label{eq:qmart}
\mathbb E_{\mu_N}\left[\Big(M^N_G(t)\Big)^2\right]=\frac{N^2}{N^{2d}}\mathbb E_{\mu_N}\left[\int^t_0 ds \sum_{\{x,y\}\in \mathcal{E}_N}(c_{x,y}(\eta_s)+c_{y,x}(\eta_s))G^2_N(x,y)\right]\leq \frac{2dt\big(1+|\alpha|\big)|G|_\infty^2}{N^d}\,.
\end{equation}
\end{remark}

\subsection{Current topology}\label{sec:currentop}

We introduce now, following the approach in \cite{BLG}, the topological setting where we can prove  a scaling limit for the current field given in \eqref{orleans}.   See \cite{KL} chapter 11 or \cite{BLG} for more details. Consider the lattice $ \mathbb{Z}^d $  endowed with the lexicographical order, consider $z\in \mathbb Z^d$ and let $ h_0=1 $, $ h_z(u)=\sqrt{2}\cos(2\pi z\cdot u) $ if $ z>0 $ and $ h_z(u)=\sqrt{2}\sin(2\pi z\cdot u) $ if $ z<0 $. In the space of real $ L^2(\Lambda) $ functions  equipped with the scalar product $ \langle f,g \rangle=\int_{\Lambda}du f(u) g(u) $,  the set $ \{h_z,\,\,z\in \mathbb{Z}^d\} $ is an orthonormal basis. Therefore given an $ L^2 $-integrable vector field $ G:\Lambda\rightarrow \mathbb{R}^d $ each component $ j\in\{1,\dots,d\} $ can be written as
\begin{equation*}
G_j=\underset{z\in \mathbb{Z}^d}{\sum} \langle G_j,h_z \rangle h_z=\underset{z\in \mathbb{Z}^d}{\sum} \mathcal{G}_j(z) h_z, \text{ with } \mathcal{G}_j(z):= \langle G_j,h_z \rangle.
\end{equation*}
 Given two $ L^2 $-vector fields $ F,G:\Lambda\rightarrow \mathbb{R}^d $, we  consider the scalar product
 \begin{equation}\label{eq:0scalarp}
 \langle F,G\rangle_0 :=\underset{j=1}{\overset{d}{\sum}}\langle F_j, G_j\rangle =  \underset{j=1}{\overset{d}{\sum}}\underset{z\in \mathbb{Z}^d}{\sum}\mathcal{F}_j(z)\mathcal{G}_j(z),
 \end{equation}
where $ \mathcal{F}_j(z)$ and $\mathcal{G}_j(z)  $ are the projections of $ F_j $ and $ G_j $ on $ h_z $. This scalar product defines the   $ L^2(\Lambda,\mathbb{R}^d) $ Hilbert space that we denote by $ \mathcal{H}^d_0 $. Consider on $ C^\infty(\Lambda, \mathbb R^d) $ the positive, symmetric linear operator $ \mathcal{L}=(1-\Delta) $. The functions $ h_z $ are its eigenvectors
\begin{equation*}
\mathcal{L}h_z= \gamma_z h_z, \text{ where \,} \gamma_z=1+4\pi^2 | z|^2.
\end{equation*}
This operator allows us to define for each $ k\geq 0 $ the Hilbert spaces $ \mathcal{H}^d_k $ obtained as the completion of $ C^\infty(\Lambda,\mathbb{R}^d) $ endowed with the scalar product $ \langle \cdot, \cdot\rangle_k $ defined by
\begin{equation}\label{eq:kscalarp}
\langle F , G\rangle_k:= \langle F, \mathcal{L}^k G\rangle_0\, ,
\end{equation}
with $F,G\in C^\infty(\Lambda,\mathbb{R}^d) $. From definition \eqref{eq:0scalarp} and properties of $\mathcal{L}  $ we have that
\begin{equation*}\label{eq:kscalar2}
\langle F , G\rangle_k= \underset{j=1}{\overset{d}{\sum}}\underset{z\in \mathbb{Z}^d}{\sum}\gamma_z^k \mathcal{F}_j(z)\mathcal{G}_j(z),
\end{equation*}
therefore for $k\geq k'\geq 0$ we have $ \mathcal{H}^d_{k} \subset\mathcal{H}^d_{k'}\subset\mathcal{H}^d_0$  because $ \mathcal{H}_k^d $ is the subspace of $ \mathcal{H}_0^d $ consisting of all vector fields $ F $ such that
\begin{equation}\label{eq:k-norm}
\|F\|^2_k= \underset{j=1}{\overset{d}{\sum}}\underset{z\in \mathbb{Z}^d}{\sum}\gamma_z^k\mathcal{F}^2_j(z)<\infty.
\end{equation}

Denote by $ \phi $ a bounded  linear functional from $ \mathcal{H}^d_k $ to $ \mathbb{R} $ belonging to the dual space $ \mathcal{H}^d_{-k}:=(\mathcal{H}^d_k)^* $,  its action on $ F\in \mathcal{H}^d_k $ is indicated with $ \phi(F) $. By Riesz representation theorem for each $ \phi\in \mathcal{H}^d_{-k} $ there is a unique $ G^\phi\in\mathcal{H}^d_{k} $ such that $\phi(F)=\langle G^\phi, F  \rangle_k $ for each $ F$ in $ \mathcal{H}^d_{k}$. From this it follows the existence of  an isometric isomorphism between $ \mathcal{H}^d_{k} $ and $ \mathcal{H}^d_{-k} $. Moreover, this isomorphism induces on $ \mathcal{H}^d_{-k} $ a scalar product $ \langle \cdot,\cdot\rangle_{-k} $, such that  given $ \phi,\phi'\in \mathcal{H}^d_{-k} $ we have $ \langle \phi,\phi'\rangle_{-k}=\langle G^\phi, G^{\phi'}\rangle_k$. This scalar product turns out to be
\begin{equation}\label{eq:scalarp-k}
\langle \phi,\phi'\rangle_{-k}=\underset{j=1}{\overset{d}{\sum}}\underset{z\in \mathbb{Z}^d}{\sum}\gamma^{-k}_z\phi(I^{j,z})\phi'(I^{j,z})=\underset{j=1}{\overset{d}{\sum}}\underset{z\in \mathbb{Z}^d}{\sum}\gamma^{k}_z \mathcal{G}_{j}^{\phi}(z)\mathcal{G}_{j}^{\phi'}(z)
\end{equation}
where  $ {I}^{j,z} $ is  the vector field such that its $ i$-th component is defined as $ I^{j,z}_i:=\delta_{i,j} h_z $ and $ \mathcal{G}_{j}^{\phi}=\langle G^\phi_j, h_z\rangle $. Therefore, the space $ \mathcal{H}^d_{-k} $ consists of all functionals such that
\begin{equation}\label{eq:-k-norm}
\|\phi\|^2_{-k}= \underset{j=1}{\overset{d}{\sum}}\underset{z\in \mathbb{Z}^d}{\sum} \gamma^{-k}_z \phi^2(I^{j,z})<\infty.
\end{equation}
Note that the space $ \mathcal{H}^d_{-k} $ can be obtained as the completion of $ \mathcal{H}^d_0 $ with respect to the scalar product $ \langle \cdot ,\cdot\rangle_{-k} $.

We will consider the  current field  $\mathcal J^N_t $ as an element of the Sobolev space $ \mathcal{H}_{-k^*}^d $, where $ k^* $  will be  determined later on, i.e. $\mathcal J^N_t \in \mathcal{H}_{-k^*}^d $. Therefore a trajectory $\left(\mathcal J^N_t\right)_{t\in[0,T]} $ will be considered to belong to the space of c\`adl\`ag trajectories $ \mathcal{D}\left([0,T],\mathcal{H}_{-k^*}^d\right) $.
Let $\mathcal J^N$ be the map from $\mathcal{D}([0,T],\Sigma_N) $ to $ \mathcal{D}\left([0,T],\mathcal{H}_{-k^*}^d\right) $ that associates to $(\eta_t)_{t\in [0,T]}$ the path $(\mathcal J^N_t)_{t\in [0,T]}$ .
We denote by $ \mathbb{P}_N $ the probability measure on $ \mathcal{D}([0,T],\mathcal{H}_{-k^*}) $ induced by $ \mathcal J^N $ and the measure $ \mu_N$, that is $ \mathbb{P}_N:= \mathbb{P}_{\mu_N}\circ (\mathcal J^N)^{-1 }$ and by $ \mathbb{E}_N $ the expectation with respect to $ \mathbb{P}_N $. With some abuse of notation we will denote also by $\mathcal J^N=\left(\mathcal J^N_t\right)_{t\in[0,T]}$
a trajectory of the current field and by $\mathcal J=\left(\mathcal J\right)_{t\in[0,T]}$ a generic element of $ \mathcal{D}\left([0,T],\mathcal{H}_{-k^*}^d\right) $.

\subsection{Hydrodynamics}
We start proving the diffusive hydrodynamic scaling behaviour of the density of our model. We have that the associated hydrodynamic equation is simply the heat equation and this is a consequence of Remark \ref{re:divj^g}, since the equation  \eqref{cont} is closed in terms of the  empirical density.  For
the law $\mathbb P_{\mu_N}$ on $D([0,T],\Sigma_N)$ we have the following result.

\begin{theorem}\label{th:hdpi}
Let $\eta_t$ be the Markov process with generator given by \eqref{generatore} with rates given in \eqref{ratej} multiplied by a factor of $N^2$. Suppose to start the process from a sequence $\mu_N$ of probability measures which are associated (according to \eqref{associato}) to a measurable density profile $\rho^*:\Lambda\to[0,1]$. Then, for any $ f\in C(\Lambda)$  and any $\varepsilon>0$, it holds
\begin{equation}\label{associatot}
\lim_{N\to +\infty}\mathbb P_{\mu_N}\left(\eta_\cdot\in\mathcal{D}([0,T],\Sigma_N):\left|\int_{\Lambda} f(u)\,\pi_N(\eta_t,du)-\int_{\Lambda}f(u)\rho_t(u)du\right|>\epsilon\right)=0\,,
\end{equation}
where $\rho_t(u)$ is the unique weak solution of the Cauchy problem
\begin{equation}\label{caud}
\left\{
\begin{array}{l}
\partial_t\rho_t(u)=\Delta\rho_t(u)\,,\forall u\in\Lambda, \forall  t>0,\\
\rho_0(u)=\rho^*(u)\,, \forall u\in\Lambda.
\end{array}
\right.
\end{equation}
\end{theorem}

\begin{proof}
Even if the model is more complex, the scaling behavior for the density can be proved similarly to the simple exclusion process (SEP). This is because a part of the instantaneous current is exactly divergence free (recall  Remark \ref{re:divj^g}) and does not contribute. We show how to reduce to the same structure of the SEP and then the proof is the same as in Chapter 4 of \cite{KL}.  	
From Dynkin's formula,  namely Lemma A.1.5.1 of \cite{KL}, for $f\in C(\Lambda)$,  we have that
\begin{equation}\label{eq:M^2densità}
M^N_{\nabla f}(t)=\int_\Lambda \pi_N(\eta_t,du)f-\int_\Lambda \pi_N(\eta_0, du)f-N^{2}\int_0^t\mathcal L_N\left(\int_\Lambda \pi_N(\eta_s, du) f\right)ds
\end{equation}
is a martingale with respect to the natural filtration (we used the notation $M^N_{\nabla f}(t)$ since for $f\in C^1(\Lambda)$ the martingale \eqref{eq:M^2densità} coincides with \eqref{martc} with $G=\nabla f$). Using Remark \ref{re:divj^g} we have that $M^N_{\nabla f}(t)$ coincides with
\begin{equation}\label{euno}
\int_\Lambda \pi_N(\eta_t, du)f-\int_\Lambda \pi_N(\eta_0, du)f- \frac{N^2}{N^d}\int_0^t\sum_{x\in \Lambda_N}
\eta_s(x)\nabla \cdot (\nabla f)_N(x)\,ds\,,
\end{equation}
where $(\nabla f)_N$ is the gradient discrete vector field $(\nabla f)_N(x,y)=f(y)-f(x)$.
Again by Lemma A.1.5.1 of \cite{KL} we have
\begin{equation}\label{edue}
\mathbb E_{\mu_N}\left[(M^N_{\nabla f}(t))^2\right]=\mathbb E_{\mu_N}\left[N^{2-2d}\int^t_0 ds\sum_{(x,y)\in E_N}c_{x,y}(\eta_s)\left(f(y)-f(x)\right)^2\right]\leq \frac{C(f,\alpha) t}{N^2}
\end{equation}
 where $ C(f,\alpha) $ is a constant depending on $f\in C(\Lambda) $  and the parameter $ \alpha $ (the above formula coincides with \eqref{eq:qmart} when $G=\nabla f$ since in that case we have $G_N(x,y)=f(y)-f(x)$).
Once obtained equations \eqref{euno} and \eqref{edue} the proof is the same as the one for SEP in Chapter 4 of \cite{KL}.
\end{proof}

\begin{remark}\label{re:weaksol}
	We recall that the unique weak solution of the Cauchy problem \eqref{caud} is also a strong solution, see \cite{KL}. Therefore the measure $ \pi_t(du)=\rho_t(u)du $ is absolutely continuous with respect to the Lebesgue measure and its density is a $ C^{1,2}([0,T]\times\Lambda) $ function.
\end{remark}

\subsection{Typical current}
We have seen above that the hydrodynamic equation can be written as a conservation law
$\partial_t\rho+\nabla\cdot J(\rho)=0$ but the typical current $J(\rho)$ does not coincide with $-\nabla \rho$ as in the classic gradient model case.
The expression of the typical current is obtained by studying the limiting behaviour of the current field. To that end,
let us introduce, for $g$ defined in \eqref{defg},
\begin{equation}\label{antidiff}
a(\rho):=E_{\nu_\rho}\left[g(\eta)\right]=2\alpha\left[\rho(1-\rho)\right]^2\,,
\end{equation}
and the antisymmetric matrix
\begin{equation}\label{Amat}
A(\rho)=\left(
\begin{array}{cc} 0 & -a'(\rho) \\ a'(\rho) & 0 \end{array} \right)\,.
\end{equation}
We have the following theorem for the current field ${\mathcal J}^N $.

\begin{theorem} \label{limit of current}
Let $\eta_t$ be the Markov process with generator given by \eqref{generatore} with rates given in \eqref{ratej} multiplied by $N^2$. Suppose to start the process from a sequence  of probability measures $\mu_N$ which are associated (according to \eqref{associato}) to a measurable density profile $\rho^*:\Lambda\to[0,1]$.  Then,  for any $C^\infty$ vector field $G$ on $\Lambda$  and for any $\epsilon >0$, it holds
\begin{equation}\label{idroJ}
\lim_{N\to+\infty}\mathbb P_{\mu_N}\left(\eta_\cdot\in\mathcal{D}([0,T],\Sigma_N): \left| \mathcal J_t^N(G)-\int_{\Lambda}du \int_0^t ds\, J(\rho_s(u))\cdot G(u)\right|>\epsilon\right)=0\,, \forall t\in[0,T],
\end{equation}
where
\begin{equation}
J(\rho)=-\nabla\rho-A(\rho)\nabla\rho\,,
\end{equation}
and $\rho_t(u)$ is the unique weak  solution of the Cauchy problem \eqref{caud} and $A(\rho)$ is given in \eqref{Amat}.
\end{theorem}

The proof of last theorem articulates in two main steps, that is, the proof of tightness for the sequence $ \{\mathbb{P}_N, N\geq 1\} $ and the characterization of its limits point. Therefore we are going to perform these two steps separately and at the end we deduce Theorem \ref{limit of current}.

\section{Proof of Theorem \ref{limit of current}}\label{PC}
As we mentioned above, the proof of the theorem relies on two main steps: tightness and the characterization of limit points. We start with the former.
\subsection{Tightness}

Let us introduce the uniform modulus of  continuity $ w_\delta({\mathcal J}) $ on $ \mathcal{D}\left([0,T], \mathcal{H}_{-k}^d\right) $ defined by
\begin{equation}\label{eq:moduluscon}
w_\delta({\mathcal J})=\underset{\substack{|t-s|\leq\delta\\0\leq s,t \leq T}}{\sup}\| {\mathcal J}_t-{\mathcal J}_s\|_{-k},
\end{equation}
with $\|\cdot\|_{-k}$ as defined in \eqref{eq:-k-norm}.

By Prokhorov's theorem, we have the following criterion for relative compactness of a sequence of probability measures $\{ \mathbb{P}_N, N\geq1 \} $ on $ \mathcal{D}\left([0,T],\mathcal{H}_{-k}^d\right)$ (see \cite{KL}  Chapter 4 Theorem 1.3 and Remark 1.4).
\begin{proposition}\label{th:funcPro}
	A sequence  of probability measures $ \{ \mathbb{P}_N,N\geq 1 \} $ defined on  $ \mathcal{D}\left([0,T],\mathcal{H}_{-k}^d\right) $ is tight if, for every $ 0\leq t\leq T $ and for every $\varepsilon>0$, we have
	\begin{enumerate}
		\item $ \underset{l\rightarrow \infty}{\lim}\underset{N\rightarrow \infty}{\limsup}\,\,\mathbb{P}_N \left( \underset{0\leq t \leq  T}{\sup}	\,\,\|{\mathcal J}_t\|_{-k}>l  \right) =0$;
		
		\item $\underset{\delta \rightarrow 0}{\lim} \,\,\,\underset{N\rightarrow \infty}{\limsup} \,\,\mathbb{P}_N\left(w_\delta ({\mathcal J})> \varepsilon \right)=0	\qquad \forall \varepsilon >0 $.
	\end{enumerate}
\end{proposition}

To prove Proposition \ref{th:funcPro} for the sequence $ \{\mathbb{P}_N , N\geq 1\}$ induced by the integrated current field,  we first derive the next lemma.

\begin{lemma}\label{le:estimate J(Ii)}
	 For every $ z $ in $ \mathbb{Z}^2$ and $ j=1,2 $,
	\begin{equation}\label{eq:lemmaJ(h)}
	\underset{N\rightarrow \infty}{\limsup}\,\,\mathbb{E}_N\left[ \underset{0\leq t\leq T}{\sup} [{\mathcal J}_t(I^{j,z}) ]^2  \right]\leq  16\pi^2 T^2\big(|z_1|+|z_2|\big)^2,
	\end{equation}
	where $ I^{j,z} $ was introduced  below \eqref{eq:scalarp-k}.
\end{lemma}

\begin{proof}
Call ${M}^N_{j,z}$ the martingale in \eqref{martc} acting on the test vector field $ G= I^{j,z} $. Since $(a+b)^2\leq 2(a^2+b^2)$, to prove inequality \eqref{eq:lemmaJ(h)}, we have to properly bound the expectation of the $ \underset{ 0\leq t \leq T}{\sup} $ of the two terms in the decomposition of  $ {\mathcal J}_t^N(I^{j,z}) $ given in \eqref{martc}.  Since $|\partial_{x_i}I^{j,z}|_\infty=2\sqrt 2\pi |z|$, from  Remark \ref{re:qmart} and Doob's inequality,  we have
\begin{equation}\label{furgoncino}
\mathbb E_{\mu_N}\left[\underset{0\leq t\leq T}{\sup}({M}^N_{j,z})^2(t)\right]\leq \frac{16d\pi^2T|z|^2\big(1+|\alpha|\big)}{N^d}\,.
\end{equation}
It remains to estimate the following expectation
\begin{align}\label{eq:estimateJ(I)}
\mathbb{E}_{\mu_N} \left\{\underset{0\leq t\leq T}{\sup} \left[ N^{2-d}\int_0^t ds\left(\sum_{x\in \Lambda_N}\eta_s(x)\nabla\cdot I^{j,z}_N(x) +\sum_{\mathfrak f\in \mathcal F_N}\tau_{\mathfrak f}g(\eta_s)\sum_{(x,y)\in f^\circlearrowleft}I^{j,z}_N(x,y)\right)\right]^2 \right\}\leq\\
T\mathbb{E}_{\mu_N}  \left[\int_0^T ds\left(N^{2-d}\sum_{x\in \Lambda_N}\eta_s(x)\nabla\cdot I^{j,z}_N(x) + N^{2-d}\sum_{\mathfrak f\in \mathcal F_N}\tau_{\mathfrak f}g(\eta_s)\sum_{(x,y)\in f^\circlearrowleft}I^{j,z}_N(x,y)\right)^2\right], \nonumber
\end{align}
the second line comes from Cauchy-Schwarz inequality. Since $ \eta $ and $g(\eta) $ are bounded, using the approximations \eqref{eq:aproxnablas}, the second line is bounded up to an infinitesimal term by
\begin{equation}\label{key1}
T^2\left(N^{-d}\sum_{x\in \Lambda_N}|\nabla\cdot I^{j,z}(x)| + N^{-d}\sum_{\mathfrak f\in \mathcal F_N}|\nabla^{\perp}\cdot I^{j,z}(y(\mathfrak f))| \right)^2
\end{equation}
where $y(\mathfrak f)$ is the center of the face $\mathfrak f$. In the above formula we have Riemann sums and by the definition of $ I^{j,z} $ formula \eqref{key1} is converging when $N\to +\infty$ to
\begin{equation*}\label{key2}
T^2\left(\sum_{i=1,2} \int_\Lambda dx |\partial_{x_i} h^z(x)|\right)^2\leq 8\pi^2 T^2\big(|z_1|+|z_2|\big)^2\,.
\end{equation*}
\end{proof}

\begin{remark}\label{glauglau}
By the formulas \eqref{eq:aproxnablas}, that are obtained by suitable Taylor expansions, we have that
$N^2\nabla\cdot I^{j,z}_N(x)=\nabla\cdot I^{j,z}(u)$ for a suitable $u$ on the face of the dual lattice centered at $x$ and $N^2\sum_{(x,y)\in f^\circlearrowleft}I^{j,z}_N(x,y)=\nabla^{\perp}\cdot I^{j,z}(u)$
for a suitable $u \in \mathfrak f$. We have therefore that the right-hand side of \eqref{eq:estimateJ(I)}
is bounded uniformly in $N, j$ by $CT^2(|z_1|+|z_2|)^2$ for a suitable constant $C$.
\end{remark}

\begin{lemma}\label{vadoascoli}
	For $ k> k^*=\frac{d}{2}+1=2 $, we have that
	\begin{equation}\label{eq:boundk*}
		\underset{N\rightarrow \infty}{\limsup}\,\,\mathbb{E}_{\mu_N}\left[ \underset{0\leq t\leq T}{\sup} \|\mathcal J^N_t\|_{-k} ^2 \right]<\infty.
	\end{equation}
\end{lemma}	
	\begin{proof}
		The expectation in \eqref{eq:boundk*} is bounded from above by
		\begin{equation*}
		\underset{j=1}{\overset{d}{\sum}}\,\underset{z\in \mathbb{Z}^d}{\sum} \,\gamma^{-k}_z \mathbb{E}_{\mu_N}\left[\underset{0\leq t\leq T}{\sup}\left[\mathcal J_t^N(I^{j,z})\right]^2\right]\,.
		\end{equation*}
Consider $k>k^*$ and use the bound  \eqref{eq:lemmaJ(h)}. By Remark \ref{glauglau} we can apply the dominated convergence theorem and we have
\begin{equation*}
\underset{N\rightarrow \infty}{\limsup}\,\, \underset{j=1}{\overset{d}{\sum}}\,\underset{z\in \mathbb{Z}^d}{\sum} \,\gamma^{-k}_z \mathbb{E}_{\mu_N}\left[\underset{0\leq t\leq T}{\sup}\left[{\mathcal J}_t^N(I^{j,z})\right]^2\right]\leq 16\pi^2 T^2 \underset{z\in \mathbb{Z}^d}{\sum}\frac{(|z_1|+|z_2|)^2}{\gamma^k_z}<+\infty\,.
\end{equation*}
\end{proof}
Using \ref{eq:boundk*} and Markov's inequality we obtain for $k>2$ the first condition of Proposition \ref{th:funcPro}.
\begin{proposition}\label{mancalabel}
For each $k> k^*=d/2+1=2 $ and each $ \epsilon>0 $
\begin{equation}\label{eq:modcon}
\underset{\delta \rightarrow 0}{\lim}\,\,\underset{N\rightarrow \infty}{\limsup}\,\,\mathbb{P}_{\mu_N}\left(\underset{\substack{|t-s|\leq\delta\\0\leq s,t \leq T}}{\sup}  \| {\mathcal J}_t^N-{\mathcal J }_s^N  \|_{-k}>\epsilon \right)=0.
\end{equation}
\end{proposition}

\begin{proof}
From Markov's inequality the probability in \eqref{eq:modcon} is bounded by	
\begin{equation*}
\frac{1}{\epsilon^2}\,\,\underset{j=1}{\overset{d}{\sum}}\,\underset{z\in \mathbb{Z}^d}{\sum} \,\gamma^{-k}_z \mathbb{E}_{\mu_N}\left[\underset{\substack{|t-s|\leq\delta\\0\leq s,t \leq T}}{\sup} \left[{\mathcal J}_t^N(I^{j,z})-{\mathcal J}_s^N(I^{j,z})\right]^2\right]\,.
\end{equation*}
We give an estimate of
\begin{equation}\label{eq:boundEjz}
\underset{\delta \rightarrow 0}{\lim}\,\,\underset{N\rightarrow \infty}{\limsup}\,\,\mathbb{E}_{\mu_N}\left[\underset{\substack{|t-s|\leq\delta\\0\leq s,t \leq T}}{\sup} \left[{\mathcal J}_t^N(I^{j,z})-{\mathcal J}_s^N(I^{j,z})\right]^2\right]\,.
\end{equation}
To that end  we start recalling the action of $ {\mathcal J}_t^N(G)-{\mathcal J}_s^N(G) $ on a test function $ G\in\mathcal{H}_{k}^d $:
\begin{equation}\label{eq:Mst}
{\mathcal J}_t^N(G)-{\mathcal J}_s^N(G)= \Big[{M}^N_G(t)-{M}^N_G(s)\Big]+\Big[N^{2-d}\int_{s}^{t} dr \underset{\{x,y\}\in\mathcal{E}_N}{\sum} j_{\eta_r}(x,y) G_N(x,y)\Big],
\end{equation}
where $ {M}^N_G(t) $ is the martingale given  in   \eqref{martc}. We bound separately the two terms inside squared parenthesis in  \eqref{eq:Mst} where $ G=I^{j,z} $. We denote by $ {M}^N_{I^{j,z}}(t-s) $ the difference of the martingales $ {M}^N_{I^{j,z}} (t) $ and $ {M}^N_{I^{j,z}} (s) $ and recall that
the uniform modulus of continuity \eqref{eq:moduluscon}
can be written as $ w_\delta({\mathcal J})=\underset{0\leq t\leq T-\delta}{\sup}\,\,\underset{|\Delta t| \leq\delta}{\sup}\|{\mathcal J}_{t+\Delta t}-{\mathcal J}_t\|_{-k} $. By Remark \ref{re:qmart} and Doob's inequality, with analogous estimates to the ones employed in \eqref{furgoncino} we get that
\begin{align*}
\underset{\delta \rightarrow 0}{\lim}\,\,\underset{N\rightarrow \infty}{\limsup}\,\,\mathbb{E}_{\mu_N}\left[\underset{\substack{|t-s|\leq\delta\\0\leq
s,t \leq T}}{\sup}   ({M}^N_{j,z})^2 (t-s) \right]\leq \underset{\delta \rightarrow 0}{\lim}\,\,\underset{N\rightarrow \infty}{\limsup}\,\,\frac{16\pi^2d\delta (1+|\alpha|)|z|^2}{N^d} =0.
\end{align*}
To treat the second term, after Chebychev's and Cauchy-Schwarz's inequalities, proceeding as we did to  bound \eqref{eq:estimateJ(I)} and using similar arguments as those in the proof of Proposition \ref{vadoascoli} we obtain
\begin{equation}
\underset{\delta \rightarrow 0}{\lim}\,\,\underset{N\rightarrow \infty}{\limsup}\,\,\mathbb{E}_{\mu_N}\left[\underset{\substack{|t-s|\leq\delta\\0\leq s,t \leq T}}{\sup} \biggl\lvert \int_{s}^{t} dr\, N^{2-d} \underset{\{x,y\}\in\mathcal{E}_N}{\sum} j_{\eta_r}(x,y) I^{j,z}_N(x,y) \biggr\rvert^2 \right]=0.
\end{equation}
From these estimates one can conclude \eqref{eq:modcon}.
\end{proof}

\begin{proof}\emph{(of Proposition \ref{th:funcPro})} Item $(1)$ is obtained from Lemma \ref{vadoascoli} and  Markov's inequality. Item $(2)$ coincides with the statement of Proposition \ref{mancalabel}.
\end{proof}

This proof  shows that the sequence $ \{\mathbb{P}_N\,\, N\geq 1 \} $ is tight on the space of trajectories $ \mathcal{D}\left([0,T],\mathcal{H}_{-k}^d\right) $.

\subsection{Characterization of limit points }

Now we characterize the unique limit points of the sequence $ \{\mathbb{P}_N\,,\, N\geq 1 \} $.

We begin by fixing some notations. Fix $ x=(x_1,x_2)\in \Lambda_{N},  \mathbb{\ell}\in\mathbb{N},\varepsilon>0,\delta>0 $. To have a simple notation, in some formulas we will write $ \varepsilon N $ even if we should instead consider its integer part $ [\varepsilon N] $.
Let us define the intervals
$$
I^{i,\ell}_p(x):=\left\{
\begin{array}{ll}
[x_i+e^{(i)},x_i+\ell  e^{(i)}] & \textrm{if}\ p=+1\,,  \\

[x_i-\ell  e^{(i)},x_i-e^{(i)}] & \textrm{if}\ p=-1\,,
\end{array}
\right.
$$
and the corresponding boxes
$$ B^\ell_{p,q}(x)= I^{1,\ell}_p(x)\times I^{2,\ell}_q(x)\subseteq \Lambda_N\,, \qquad  p,q\in \{1,-1\}\,. $$

This means that along the 4 possible values of the indexes $ p,q $ we are considering the four boxes of size $ \ell/N $ having $ x $ as a corner. The point $x$ does not belong to the boxes to make them disjoint, and this will be important in the proof below. We define also
\begin{equation}\label{eq:boxdensity}
 \eta^{(p,q)}_\ell(x):= \frac{1}{\ell^2}\underset{y\in B^\ell_{p,q}(x) }{\sum}\,\eta(y)
\end{equation}
the particles density in the box $  B^\ell_{p,q}(x) $.   We consider four approximations of the identity; consider on the continuous torus $ u=(u_1,u_2), v=(v_1,v_2)\in \Lambda$ and define
\begin{equation}\label{eq:characteristic p,q}
i^{(p,q,u)}_{\varepsilon}(v):=\left\{
\begin{array}{ll}
\varepsilon^{-2}1_{[u_1,u_1+\varepsilon)\times [u_2,u_2+\varepsilon)}(v) \text{ if } (p,q)=(1,1),\\
\varepsilon^{-2}1_{[u_1,u_1+\varepsilon)\times (u_2-\varepsilon,u_2]}(v) \text{ if } (p,q)=(1,-1),\\
\varepsilon^{-2}1_{(u_1-\varepsilon,u_1]\times [u_2, u_2+\varepsilon)}(v) \text{ if } (p,q)=(-1,+1),\\
\varepsilon^{-2}1_{(u_1-\varepsilon,u_1)\times (u_2-\varepsilon,u_2]}(v) \text{ if } (p,q)=(-1,-1).\\
\end{array}\right.
\end{equation}
We use also the shortcuts
\begin{equation*}
 \pi_t\left( f\right) :=\int _\Lambda f(u) \rho_t(u)du,\,\, \pi^N_t(f):=\int_\Lambda f(u)\, \pi^N(\eta_t,du),
\end{equation*}
where $ \rho_t(u) $ solves the Cauchy problem \eqref{caud} and $f\in C^2(\Lambda)$. We associate to each vertex $ x\in \Lambda_N $ the non-oriented face $ \mathfrak{f}_x=\{x,x+e^{(1)},x+e^{(1)}+e^{(2)},x+e^{(2)}\} $; accordingly  $ f ^\circlearrowleft_x$ and $ f^\circlearrowright_x $ are the corresponding anticlockwise and clockwise orientations of $\mathfrak{f}_x$.

\begin{proposition}\label{pr:char}
Let $ \mathbb{P} $ be  a limit point of the sequence $ \{ \mathbb{P}_N\,,\, N\geq 1\} $. Then, for $ k>k^* $,
\begin{equation}\label{eq:char}
\mathbb{P}\Big( \mathcal J_\cdot\in \mathcal{C}([0,T],\mathcal{H}_{-k}) \ \textrm{and}\  \mathcal F(G,t,\rho)=0 \,\, \,\,\forall t\in [0,T] \Big)=1\,, \qquad \, \forall \, G\in C^\infty(\Lambda; \mathbb R^2)\,,
\end{equation}
where
\begin{equation}\label{eq:F}
 \mathcal F(G,t,\rho)=\mathcal J_{t} (G)- \int_0^tds\int_{\Lambda}du \,\,\left(\rho_s(u)\nabla \cdot G(u)   +a\left(\rho_s(u)\right)\nabla^{\perp}\cdot G(u)\right),
\end{equation}
with $ \rho_t(u) $ solving the Cauchy problem \eqref{caud} and $a(\cdot)$ is defined in \eqref{antidiff}.
\end{proposition}

\begin{proof}
	Condition (2) in Proposition \ref{th:funcPro} tells us that the limit points are concentrated on continuous paths, i.e. paths  in   $ \mathcal{C}\left([0,T],\mathcal{H}_{-k}^d\right)$.
	
	It remains to show that for any $ \delta>0  $ and  any vector field $G\in C^\infty(\Lambda; \mathbb R^2)$
	
\begin{equation}\label{eq:supF}
\mathbb{P}\left( \mathcal J_\cdot\in \mathcal{C}([0,T],\mathcal{H}_{-k}) :\underset{0\leq t\leq T}{\sup} | \mathcal F(G,t,\rho)|>\delta \right)=0.
\end{equation}
By Lebesgue's differentiation theorem $ \underset{\varepsilon\rightarrow 0}{\lim}\,\, \pi_s\left( i^{(p,q,u)}_\varepsilon\right) = \rho_s(u) $ for all  $ u\in\Lambda $, for any $p,q$ and for almost every $s\in [0,T]$. Since $G\in C^\infty(\Lambda; \mathbb R^2)$ by dominated convergence Theorem we have therefore that
	\begin{align}\label{eq:2rho}
& \lim_{\varepsilon \to 0}\underset{0\leq t\leq T}{\sup} \left|  \int_0^t ds \int _\Lambda du \,\, a\left(\rho_s(u)\right)\nabla^{\perp}\cdot G(u) \right.  \\
& - \int_0^t  ds\int_\Lambda du\, \left(\alpha \, \pi_s\left( i^{(-1,-1,u)}_\varepsilon\right) \left(1- \pi_s \left(i^{(1,-1,u)}_\varepsilon\right) \right)  \pi_s \left(i^{(1,1,u)}_\varepsilon\right)
\left(1-\pi_s\left( i^{(-1,1,u)}_\varepsilon\right)\right)
\right)\nabla^{\perp}\cdot G(u) \nonumber\\
& - \int_0^t  ds  \int_\Lambda du\, \left(\alpha \left( 1 -  \pi_s\left( i^{(-1,-1,u)}_\varepsilon\right)\right) \pi_s\left( i^{(1,-1,u)}_\varepsilon\right)\right.\nonumber\\
& \left.\left.  \left(1- \pi_s\left( i^{(1,1,u)}_\varepsilon\right) \right)
\pi_s\left(i^{(-1,1,u)}_\varepsilon\right)
\right)\nabla^{\perp}\cdot G(u)\right|=0\,. \nonumber
\end{align}

By summing and subtracting proper terms and using the above formula, we have that \eqref{eq:supF} is deduced by proving for any $\delta >0$ that

 \begin{align}\label{eq:scatole}
& \underset{\varepsilon\rightarrow 0}{\liminf}\, \mathbb{P}\left(\underset{0\leq t\leq T}{\sup} \left|  \mathcal J_t(G)- \int_0^t\,  ds\,\pi_s\left(\nabla \cdot G\right)\right.  \right.\\
 & - \int_0^t ds\int_\Lambda du\, \left(\alpha\,  \pi_s\left( i^{(-1,-1,u)}_\varepsilon\right)\left(1- \pi_s\left( i^{(1,-1,u)}_\varepsilon \right)\right)  \pi_s \left(i^{(1,1,u)}_\varepsilon\right)
\left(1-\pi_s \left( i^{(-1,1,u)}_\varepsilon \right)\right)
  \right)\nabla^{\perp}\cdot G(u) \nonumber\\
  & - \int_0^t  ds  \int_\Lambda du\, \left(\alpha \left( 1 -  \pi_s\left( i^{(-1,-1,u)}_\varepsilon\right)\right) \pi_s\left( i^{(1,-1,u)}_\varepsilon\right)\right.\nonumber\\
 & \left.\left.\left.  \left(1- \pi_s\left( i^{(1,1,u)}_\varepsilon\right) \right)
   \pi_s\left(i^{(-1,1,u)}_\varepsilon\right)
  \right)\nabla^{\perp}\cdot G(u)\right|>\delta\right)=0\,. \nonumber
\end{align}
By Portmanteau's Theorem we can bound from above the limit \eqref{eq:scatole} by
\begin{align}\label{eq:Port}
& \underset{\varepsilon\rightarrow 0}{\liminf}\underset{N\rightarrow +\infty}{\liminf}\,\mathbb{P}_{\mu_N}\left(\underset{0\leq t\leq T}{\sup} \left|  \mathcal J^N_t(G)- \int_0^t \, ds\,\pi_s\,\left(\nabla \cdot G\right)\right.  \right.\\
& - \int_0^t \,\, ds\int_\Lambda du\, \left(\alpha  \pi_s\left( i^{(-1,-1,u)}_\varepsilon\right)\left(1- \pi_s\left( i^{(1,-1,u)}_\varepsilon\right) \right)  \pi_s \left(i^{(1,1,u)}_\varepsilon\right)
\left(1-\pi_s\left( i^{(-1,1,u)}_\varepsilon\right)
\right)\right)\nabla^{\perp}\cdot G(u) \nonumber\\
& - \int_0^t  ds  \int_\Lambda du\, \left(\alpha \left( 1 -  \pi_s \left(i^{(-1,-1,u)}_\varepsilon\right)\right) \pi_s\left( i^{(1,-1,u)}_\varepsilon\right)\right. \nonumber \\
&\left.\left. \left. \left(1- \pi_s\left( i^{(1,1,u)}_\varepsilon\right) \right)
 \pi_s \left( i^{(-1,1,u)}_\varepsilon\right)
\right)\nabla^{\perp}\cdot G(u)\right|>\delta\right). \nonumber
\end{align}
We sum and subtract $ (1/2)\int_{0}^{t} ds \underset{(x,y)\in E_N}{\sum} j_{\eta_s}(x,y)G_N(x,y) $ to the term inside the supremum in \eqref{eq:Port}. Recalling \eqref{martc} and \eqref{eq:martc2},  we bound the probability in \eqref{eq:Port} by the sum of the next three terms

\begin{equation}\label{eq:Pmart}
\mathbb{P}_{\mu_N} \left(   \underset{0\leq t\leq T}{\sup} \left| M^N_G(t)  \right|> \frac{\delta}{3}\right),
\end{equation}

\begin{equation}\label{eq:Pemp}
\mathbb{P}_{\mu_N} \left(   \underset{0\leq t\leq T}{\sup} \left| N^{2}\int_0^t ds\,\pi_s^N\left(\nabla\cdot G_N\right) -\int_0^t  ds\,\pi_s\left(\nabla \cdot G\right) \right|> \frac{\delta}{3}\right),
\end{equation}

and

\begin{align}\label{eq:Pre}
& \mathbb{P}_{\mu_N}\left(\underset{0\leq t\leq T}{\sup} \left| \int_{0}^{t} ds \sum_{\mathfrak f\in \mathcal F_N}\tau_{\mathfrak f}g(\eta_s)\sum_{(x,y)\in f^\circlearrowleft}G_N(x,y)\right.  \right.\\
& - \int_0^t ds\int_\Lambda du\, \left(\alpha  \pi_s \left(i^{(-1,-1,u)}_\varepsilon\right)\left(1- \pi_s \left(i^{(1,-1,u)}_\varepsilon\right) \right)  \pi_s\left( i^{(1,1,u)}_\varepsilon\right)
\left(1- \pi_s \left(i^{(-1,1,u)}_\varepsilon\right) \right)
\right)\nabla^{\perp}\cdot G(u) \nonumber\\
& - \int_0^t  ds  \int_\Lambda du\, \left(\alpha \left( 1 -  \pi_s\left( i^{(-1,-1,u)}_\varepsilon\right)\right) \pi_s\left( i^{(1,-1,u)}_\varepsilon\right)  \right.\nonumber\\
&\left.\left.\left. \left(1-\pi_s\left( i^{(1,1,u)}_\varepsilon\right) \right)
 \pi_s \left(i^{(-1,1,u)}_\varepsilon\right)
\right)\nabla^{\perp}\cdot G(u)\right|>\frac{\delta}{3}\right). \nonumber
\end{align}

From Doob's inequality and \eqref{eq:qmart}  the probability in \eqref{eq:Pmart} vanishes as $ N\rightarrow \infty $. The same holds  for the probability in \eqref{eq:Pemp} by the approximation \eqref{eq:aproxnablas} for the discrete divergence and the law of large numbers for the empirical density (see Theorem \ref{th:hdpi}). Again by the law of large numbers for the density, to show that \eqref{eq:Pre} is converging to zero for any $\delta$, we can simply show that

\begin{align}\label{eq:PreN}
& \mathbb{P}_{\mu_N}\left(\underset{0\leq t\leq T}{\sup} \left| \int_{0}^{t}ds \sum_{\mathfrak f\in \mathcal F_N}\tau_{\mathfrak f}g(\eta_s)\sum_{(x,y)\in f^\circlearrowleft}G_N(x,y)\right.  \right.\\
& - \int_0^t ds\int_\Lambda du\, \left(\alpha  \pi^N_s \left(i^{(-1,-1,u)}_\varepsilon\right)\left(1- \pi^N_s \left(i^{(1,-1,u)}_\varepsilon\right) \right)  \pi^N_s\left( i^{(1,1,u)}_\varepsilon\right)
\left(1- \pi^N_s \left(i^{(-1,1,u)}_\varepsilon\right) \right)
\right)\nabla^{\perp}\cdot G(u) \nonumber\\
& - \int_0^t  ds  \int_\Lambda du\, \left(\alpha \left( 1 -  \pi^N_s\left( i^{(-1,-1,u)}_\varepsilon\right)\right) \pi^N_s\left( i^{(1,-1,u)}_\varepsilon\right) \right. \nonumber\\
& \left.\left.\left.\left(1- \pi^N_s\left( i^{(1,1,u)}_\varepsilon\right) \right)
\pi^N_s \left(i^{(-1,1,u)}_\varepsilon\right)
\right)\nabla^{\perp}\cdot G(u)\right|>\tilde{\delta}\right), \nonumber
\end{align}
is converging to zero when $N\to+\infty$, for any $ \tilde{\delta} $. Recalling  \eqref{defg}, the probability in  \eqref{eq:PreN} can be bounded by the sum of the following two terms

\begin{align}\label{eq:Preg1}
& \mathbb{P}_{\mu_N}\left(\underset{0\leq t\leq T}{\sup} \left| \int_{0}^{t} ds \sum_{x\in \Lambda_N}   \left\{\eta_s(x)\left(1-\eta_s\left(x+e^{(1)}\right)\right) \right.\right.\right. \\
&\left. \eta_s \left(x+e^{(1)}+e^{(2)}\right)\left(1-\eta_s\left(x+e^{(2)}\right)\right)\right\}\sum_{(w,z)\in f_x^\circlearrowleft}G_N(w,z)\nonumber\\
& - \int_0^t ds\int_\Lambda du\, \left( \pi^N_s \left(i^{(-1,-1,u)}_\varepsilon\right)\left(1- \pi^N_s \left(i^{(1,-1,u)}_\varepsilon\right) \right) \right.\nonumber\\
 &\left.\left.\left. \pi^N_s\left( i^{(1,1,u)}_\varepsilon\right)
\left(1- \pi^N_s \left(i^{(-1,1,u)}_\varepsilon\right) \right)
\right)\nabla^{\perp}\cdot G(u)\right|>\frac{\tilde \delta}{2|\alpha|}\right) \nonumber,
\end{align}
and
\begin{align}{}\label{eq:Preg2}
& \mathbb{P}_{\mu_N}\left(\underset{0\leq t\leq T}{\sup} \left|\int_{0}^{t} ds  \sum_{x\in \Lambda_N} \left\{\left(1-\eta_s(x)\right)\eta_s\left(x+e^{(1)}\right)\right.\right.\right. \\
&\left.\left(1-\eta_s\left(x+e^{(1)}+e^{(2)}\right)\right) \eta_s\left(x+e^{(2)}\right) \right\}\sum_{(w,z)\in f_x^\circlearrowleft}G_N(w,z) \nonumber\\
& - \int_0^t  ds  \int_\Lambda du\, \left( \left( 1 -  \pi^N_s\left( i^{(-1,-1,u)}_\varepsilon\right)\right) \pi^N_s\left( i^{(1,-1,u)}_\varepsilon\right)\right. \nonumber \\
& \left.\left.\left.\left(1-\pi^N_s\left( i^{(1,1,u)}_\varepsilon\right) \right)
\pi^N_s \left(i^{(-1,1,u)}_\varepsilon\right)
\right)\nabla^{\perp}\cdot G(u)\right|>\frac{\tilde{\delta}}{2|\alpha|}\right). \nonumber
\end{align}
By the approximation \eqref{eq:aproxnablas}, we can replace $ \sum_{(w,z)\in f^\circlearrowleft_x}G_N(w,z) $ by $ \nabla^\perp \cdot G(x)/N^2 $ for $N$ large. Moreover by the definitions we have
\begin{align}\label{eq:replaceboxes}
\Big|\pi^N_s\left(i^{(p,q,x)}_\varepsilon\right)-\eta^{(p,q)}_{s,\,\varepsilon N}(x)\Big|\leq \frac{2}{\varepsilon N}\, ,
\qquad x\in \Lambda_N,
\end{align}
we can bound \eqref{eq:Preg1} and \eqref{eq:Preg2}, for $N$ large enough, respectively by
\begin{align}\label{eq:Preplaced1}
& \mathbb{P}_{\mu_N}\left(\underset{0\leq t\leq T}{\sup} \left| \int_{0}^{t} ds\, \frac{1}{N^2}\sum_{x\in \Lambda_N}   \left\{   \nabla^\perp \cdot G\left(x\right)\right.\right.\right. \\
&  \left[ \eta_s(x)\left(1-\eta_s\left(x+e^{(1)}\right)\right) \eta_s \left(x+e^{(1)}+e^{(2)}\right)\left(1-\eta_s\left(x+e^{(2)}\right)\right)\right.\nonumber\\
& -\left.\left.\left.\left.   \, \eta^{(-1,-1)}_{s,\,\varepsilon N}(x)\left(1- \eta^{(1,-1)}_{s,\,\varepsilon N}(x) \right)  \eta^{(1,1)}_{s,\,\varepsilon N}(x)
\left(1- \eta^{(-1,1)}_{s,\,\varepsilon N}(x) \right)
\right]\right\}\right|>\hat{\delta}\right) \nonumber,
\end{align}
and
\begin{align}\label{eq:Preplaced2}
& \mathbb{P}_{\mu_N}\left(\underset{0\leq t\leq T}{\sup} \left| \int_{0}^{t}\,ds  \frac{1}{N^2}\sum_{x\in \Lambda_N}\left\{\nabla^{\perp}\cdot G\left(x\right)\right.\right.\right.\\ &\left[\left(1-\eta_s(x)\right)\eta_s\left(x+e^{(1)}\right)\left(1-\eta_s\left(x+e^{(1)}+e^{(2)}\right)\right) \eta_s\left(x+e^{(2)}\right) \right. \nonumber\\
& - \left.\left.\left.\left.  \left( 1 -  \eta^{(-1,-1)}_{s,\,\varepsilon N}(x)\right) \eta^{(1,-1)}_{s,\,\varepsilon N}(x)  \left(1-\eta^{(1,1)}_{s,\,\varepsilon N}(x) \right)
\eta^{(-1,1)}_{s,\,\varepsilon N}(x)\right]\right\}
\right|>\hat{\delta}\right). \nonumber
\end{align}
for a suitable $ \hat{\delta}<\tilde{\delta}/2|\alpha| $. The key result that allows to conclude the proof is Proposition \ref{pr:repla}, together with Markov's inequality, implying that the probabilities in  \eqref{eq:Preplaced1} and \eqref{eq:Preplaced2}  vanish as $ N\rightarrow\infty $ and $ \varepsilon\rightarrow 0 $. This ends the proof.

\end{proof}

We remark that \eqref{eq:F} is a weak form of $ \mathcal J_t(G)-\int_0^Tdt \int_{\Lambda}du \, J(\rho_t(u))\cdot G(u) $ with $ J(\rho)=-\nabla\rho-A(\rho)\nabla\rho $, but from the regularity property of $ \rho_t(u) $ discussed in Remark \ref{re:weaksol} we have that the two forms are equivalent. Hence the uniqueness  and characterization of the limit point follows from this and the fact that at time $ 0 $ we have $ \mathcal J^N_0(G)=0 $. Therefore the proof of Theorem \ref{limit of current} is completed once we show the auxiliary   replacement lemma used in the proof of Proposition \ref{pr:char}.

\subsection{Replacement lemma}
In this section we discuss how to prove the replacement lemma used to deduce that \eqref{eq:Preplaced1} and \eqref{eq:Preplaced2} converge to zero  when $N\to +\infty$ and $\varepsilon \to 0$.
We start to define the Dirichlet form and  the Carr\'e du Champ operator and we will discuss a relation between them.

\subsubsection{Dirichlet forms}

Recall that the  Bernoulli product  measure $$ \nu_\rho(\eta)=\underset{x\in \Lambda_N}{\prod}\rho^{\eta(x)}(1-\rho)^{1-\eta(x)},$$ is invariant for the dynamics. Let $ f:\Sigma_N\rightarrow\mathbb{R} $ be a density with respect to $ \nu_\rho $. The Dirichlet form of the process is defined as
\begin{equation}
-\left\langle \mathcal{L}_N\sqrt{f},\sqrt{f}\right\rangle_{\nu_\rho},\,\,\text{ with } \langle g,h\rangle_{\nu} :=\int \nu(d\eta)g(\eta)h(\eta)=E_{\nu}(gh),
\end{equation}
for all functions $ g,h:\Sigma_N\rightarrow \mathbb{R} $ and $\nu$ a probability measure in $\Sigma_N$. Moreover, we define the quadratic form, with respect to $ \nu_\rho $, as the operator  $ \mathfrak{D}_N  $ acting on positive functions $ f:\Sigma_N\rightarrow\mathbb{R} $ as follows,
\begin{equation}
\mathfrak{D}_N\left(\sqrt{f},\nu_\rho\right):= \frac{1}{2}\underset{(x,y)\in E_N}{\sum}\int\nu_\rho (d\eta) c_{x,y}(\eta)\left(\sqrt{f}(\eta^{x,y})-\sqrt{f}(\eta)\right)^2.
\end{equation}
A direct computation, using the invariance of $\nu_\rho$ and the fact that $ \nu_\rho(\eta^{x,y})/\nu_\rho(\eta)=1 $,  tells us that the Dirichlet form and the quadratic form coincide, i.e.
\begin{equation}\label{eq:carrè=dirichlet}
	-\left\langle \mathcal{L}_N\sqrt{f},\sqrt{f}\right\rangle_{\nu_\rho}= \mathfrak{D}_N(\sqrt{f},\nu_\rho).
	\end{equation}

\subsubsection{Replacement lemma on the discrete torus}

First we prove a replacement lemma and then show how to apply the basic lemma to our specific case.
Consider $ \psi_N:\Sigma_N\to \mathbb R $ a bounded function whose domain does not overlap the vertex $ 0 $ nor the box $ B^\ell_{p,q}(0) $ for any $\ell\in \mathbb N$.
Let us define
$$ V^{N,p,q}_{\ell,G}(\eta):=\frac{1}{N^2}\underset{x\in \Lambda_N}{\sum}\left(\eta(x)-\eta^{(p,q)}_{\ell}(x)\right)\,\,\tau_x\psi_N(\eta)   \nabla^\perp\cdot G(x)  \,, $$
where $G\in C^\infty(\Lambda; \mathbb R^2)$.
We have the following

\begin{lemma}\label{le:repla}
	Let $ \left\{\psi_N:\Sigma_N\to \mathbb R, N\geq 1\right\} $ be  a uniformly  bounded sequence of  functions whose domains do not overlap the vertex $ 0 $ nor the box $ B_{p ,q }^{\ell}(0) $ for any $\ell\in \mathbb N$. For $ \ell=\varepsilon N $ we have that
	\begin{equation}\label{eq:replacement}
	\underset{\varepsilon\rightarrow 0}{\overline{\lim}}\,\,
	\underset{N\rightarrow\infty}{\overline{\lim}}	\mathbb{E}_{\mu_N}\left[\left|\int^t_0 \,ds\, V^{N,p,q}_{\varepsilon N, G}(\eta_s)    \right|\right]=0\,.
	\end{equation}
\end{lemma}

The indexes $ (p,q) $ are fixed and recall that $p,q\in \{-1,1\} $.

\begin{proof}
	By the entropy inequality, see for example Section A1.8 in \cite{KL}, the expectation in \eqref{eq:replacement} can be bounded by
	\[ \frac{H(\mu_N|\nu_\rho)}{N^2 B}+\frac{1}{N^2 B}\log {\mathbb E}_{\nu_\rho}\left( \exp \left|  BN^2 \int^t_0 \,ds\,V^{N,p,q}_{\ell, G}(\eta_s)\right|\right),   \]
	where $ \nu_\rho $ is the Bernoulli measure of parameter $ \rho $  and $ B $ is an arbitrary positive constant. From Feynman-Kac's formula and the variational formula for the largest eigenvalue of a symmetric operator (see respectively Proposition A1.7.1  and Lemma A1.7.2 in \cite{KL}) we can bound last expression from above by
	
	\begin{equation}\label{eq:feynman}
	\frac{H(\mu_N|\nu_\rho)}{N^2 B}+ t\,\, \underset{f}{\sup} \left\{\left\langle V^{N,p,q}_{\ell,G},f\right\rangle_{\nu_\rho}+\frac{1}{B}\left\langle \mathcal{L}_N\sqrt{f},\sqrt{f}\right\rangle_{\nu_\rho}\right\}
	\end{equation}
	where the supremum is carried over all densities $ f $ with respect to $ \nu_\rho $. Note that even if the generator $\mathcal L_N$ is not reversible we have the bound \eqref{eq:feynman}, see the comments
	on Section A.1.7 in \cite{KL}.
	The relative entropy is bounded from above by $ c N^2 $, where $ c $ is a positive constant, see Theorem A.1.8.6 in \cite{KL}. We have then that \eqref{eq:feynman} is bounded from above by
	\begin{equation}\label{eq:feynman2}
	\frac{c}{B}+t\,\,\underset{f}{\sup} \left\{\left\langle V^{N,p,q}_{\ell,G},f\right\rangle_{\nu_\rho}-\frac{1}{B}\mathfrak{D}_N(\sqrt{f},\nu_\rho)\right\}\,.
	\end{equation}
	We consider the following telescopic expansion
	  \begin{equation*}
	\eta(x)-\eta^{(p,q)}_{\ell}(x)= \frac{1}{\ell^2}\underset{\left\{y\in B^\ell_{p,q}(x)\right\}}{\sum}\,\underset{\left\{z^{(i)}\in \gamma_{x,y}\right\}}{\sum}\left(\eta\left(z^{(i)}\right)-\eta\left(z^{(i+1)}\right)\right),
	\end{equation*}
	 where $ \gamma_{x,y} $ is the minimal length path from $ x $ to $ y $, with the final vertex $y$ removed, obtained going from $x$ to
	 $y$ walking first in the direction $pe^{(1)}$ until we cross the perpendicular line containing $y$ and then
	 walking in the direction $qe^{(2)}$ until we reach $y$. The final vertex $y\not\in \gamma_{x,y}$ in such a way that $\underset{\left\{z^{(i)}\in \gamma_{x,y}\right\}}{\sum}\left(\eta\left(z^{(i)}\right)-\eta\left(z^{(i+1)}\right)\right)=\eta(x)-\eta(y)$.
	 Using the above telescopic formula, the change of variables $ \eta'=\eta^{z^{(i)},z^{(i+1)}} $ and the hypothesis on the domain of $ \psi_N $,  we get that $ \left\langle V^{N,p,q}_{\ell,G},f\right\rangle_{\nu_\rho} $ is equal to
	\begin{multline*}
	\frac{1}{2\ell^2}\int \nu_\rho(d\eta)\frac{1}{N^2}\underset{x,y,i}{\sum}\left(\eta\left(z^{(i)}\right)-\eta\left(z^{(i+1)}\right)\right)\tau_x\psi_N(\eta) \nabla^\perp\cdot G(x)\cdot\nonumber\\
	\Big(\sqrt{f}(\eta)-\sqrt{f}\left(\eta^{z^{(i)},z^{(i+1)}}\right)\Big)\Big(\sqrt{f}(\eta)+\sqrt{f}\left(\eta^{z^{(i)},z^{(i+1)}}\right)\Big)\,,
	\end{multline*}
	where we write shortly $\underset{x,y,i}{\sum}$ to denote $\underset{x\in \Lambda_N}{\sum}\underset{\left\{y\in B^\ell_{p,q}(x)\right\}}{\sum}\,\underset{\left\{z^{(i)}\in \gamma_{x,y}\right\}}{\sum}$. We  assume  that $\nabla^\perp\cdot G(x)\neq 0,$ otherwise the integral above is null and there is nothing to prove. We also
assume that $ c_{z^{(i)},z^{(i+1)}}(\eta)>0$, because if that is not the case then the factors $\left(\eta\left(z^{(i)}\right)-\eta\left(z^{(i+1)}\right)\right)$ are equal to zero and again there is nothing to prove. 	Applying Young's inequality, last expression can  be bounded from above by
	\begin{multline}\label{Young}
	\frac{1}{4\ell^2}\underset{x,y,i}{\sum}\left[\frac{1}{N^2} \nabla^\perp\cdot G(x)\right]\cdot\left\{\int{\nu_\rho(d\eta)}A^\eta_{x,y,i}   \left(\sqrt{f}(\eta)-\sqrt{f}\left(\eta^{z^{(i)},z^{(i+1)}}\right)\right)^2+\right.\\
	\left.\int{\nu_\rho(d\eta)} \frac{1}{A^\eta_{x,y,i}}\left(\eta(z^{(i)})-\eta(z^{(i+1)})\right)^2 \left(\tau_x\psi_N(\eta)\right)^2\left(\sqrt{f}(\eta)+\sqrt{f}\left(\eta^{z^{(i)},z^{(i+1)}}\right)\right)^2\right\},
	\end{multline}
	where we choose $ A_{x,y,i}^\eta:=\frac{4\left(c_{z^{(i)},z^{(i+1)}}(\eta)\right)}{B c' \ell\left[\frac{1}{N^2} \nabla^\perp\cdot G(x)\right]}$, with $ c' $ a suitable positive constant.
	
	 The first term in \eqref{Young} is equal to
	\begin{equation}\label{eq:this should Dirichlet}
	   \frac{1}{Bc'\ell^3}\underset{x,y,i}{\sum}\int{\nu_\rho(d\eta)}c_{z^{(i)},z^{(i+1)}}(\eta) \left(\sqrt{f}(\eta)-\sqrt{f}\left(\eta^{z^{(i)},z^{(i+1)}}\right)\right)^2,
	\end{equation}
	that can be bounded (here it is relevant the constant $c'$ that is used in a simple counting argument that we omit) by  $ \frac{1}{B}\underset{(x,y)\in E_N}{\sum}\int{\nu_\rho(d\eta)}c_{x,y}(\eta)\left (\sqrt{f}(\eta)-\sqrt{f}(\eta^{x,y})\right)^2$, which cancels in \eqref{eq:feynman2} with $-\frac{1}{B}\mathfrak{D}_N(\sqrt{f},\nu_\rho)$.

	Since $ \frac{1}{N^2} \nabla^\perp\cdot G(x)
	= O\left(\frac{1}{N^2}\right) $, the second term in \eqref{Young} is bounded from above by $ \frac{CB\ell^2}{N^2} $, where $ C $ is  a positive constant. Considering $ \ell=\varepsilon N $ we have that \eqref{eq:feynman2} is smaller or equal than
	\begin{equation}\label{jbondf}
	\frac{c}{B}+CB\varepsilon^2, \,
	\end{equation}
	therefore taking the limits in \eqref{jbondf} first in $N\to \infty$, then in $\varepsilon\to 0$ and finally in $B\to+\infty$, we obtain the result.
	
\end{proof}

Using this basic lemma we can finally prove the following result

\begin{proposition}\label{pr:repla}{[Replacement lemma]}
	
	Let $ G:\Lambda\rightarrow \mathbb{R}^2 $ be  a $ C^\infty(\Lambda) $ vector field. For any $ t\in[0,T] $, we have that
	\begin{align}\label{eq:repg1}
	&	\underset{\varepsilon\rightarrow 0}{\overline{\lim}}\,\,
	\underset{N\rightarrow\infty}{\overline{\lim}}\mathbb{E}_{\mu_N}\left( \left| \int_{0}^{t} ds\, \frac{1}{N^2}\sum_{x\in \Lambda_N}   \left\{   \nabla^\perp \cdot G\left(x\right)\left[\tau_x g_1(\eta_s)\right.\right.\right.\right.\\
	& \left.\left.\left.\left. -   \, \eta^{(-1,-1)}_{s,\,\varepsilon N}(x)\left(1- \eta^{(1,-1)}_{s,\,\varepsilon N}(x) \right)  \eta^{(1,1)}_{s,\,\varepsilon N}(x)
	\left(1- \eta^{(-1,1)}_{s,\,\varepsilon N}(x) \right)
	\right]\right\}\right|\right)=0 \nonumber,
	\end{align}
	and
	\begin{align}\label{eq:repg2}
	&	\underset{\varepsilon\rightarrow 0}{\overline{\lim}}\,\,
	\underset{N\rightarrow\infty}{\overline{\lim}}\mathbb{E}_{\mu_N}\left( \left| \int_{0}^{t} ds\, \frac{1}{N^2}\sum_{x\in \Lambda_N}   \left\{   \nabla^\perp \cdot G\left(x\right)\left[\tau_x g_2(\eta_s)\right.\right.\right.\right.\\
	& \left.\left.\left.\left. -   \, \left( 1 -  \eta^{(-1,-1)}_{s,\,\varepsilon N}(x)\right) \eta^{(1,-1)}_{s,\,\varepsilon N}(x)  \left(1-\eta^{(1,1)}_{s,\,\varepsilon N}(x) \right)
	\eta^{(-1,1)}_{s,\,\varepsilon N}(x)
	\right]\right\}\right|\right)=0 \nonumber,
	\end{align}
	where $$ g_1(\eta)=\eta(x)\left(1-\eta\left(x+e^{(1)}\right)\right) \eta \left(x+e^{(1)}+e^{(2)}\right)\left(1-\eta\left(x+e^{(2)}\right)\right) $$ and $$ g_2(\eta)= \left(1-\eta(x)\right)\eta\left(x+e^{(1)}\right)\left(1-\eta\left(x+e^{(1)}+e^{(2)}\right)\right) \eta\left(x+e^{(2)}\right).$$
\end{proposition}
In the present context, the replacement formulas \eqref{eq:repg1} and \eqref{eq:repg2} are done by using as auxiliary measure the Bernoulli product measure $ \nu_\rho $ of constant profile,  therefore we  can  replace the occupation variables with the average density on boxes of side $ \ell=\varepsilon N $ (see for example \cite{KL}). To prove Proposition \ref{pr:repla} we have to apply Lemma \ref{le:repla} several times. For example in $ g_1(\eta) $, the full proof would ask  the following steps:

\vspace{0.5cm}

\begin{itemize}
	\item[1)] Replace $ \eta(x)\tau_x\psi_N(\eta) $  with
	$ \eta^{(-1,-1)}_{\varepsilon N}(x)\tau_x\psi_N(\eta) $;
	\item[2)] Replace $\left(1- \eta\left(x+e^{(1)}\right)\right) \tau_x\psi_N(\eta) $  with
	$ \left(1- \eta^{(1,-1)}_{\varepsilon N}(x) \right)\tau_x\psi_N(\eta) $;
	\item[3)] Replace  $ \eta\left(x+e^{(1)}+e^{(2)}\right)\tau_x\psi_N(\eta) $ with
	$ \eta^{(1,1)}_{\varepsilon N}(x)  \tau_x\psi_N(\eta) $;
	\item[4)] Replace   $ \left(1- \eta\left(x+e^{(2)}\right)\right)\tau_x\psi_N(\eta) $ with
	$ \left(1-\eta^{(-1,1)}_{\varepsilon N}(x)\right)\tau_x\psi_N(\eta) $;

\end{itemize}
\medskip
where in 1), 2), 3) and 4) the function $ \psi_N(\eta) $ is given  respectively  by
\medskip
\begin{itemize}
	\item [1)] $ \left(1-\eta\left(x+e^{(1)}\right)\right) \eta \left(x+e^{(1)}+e^{(2)}\right)\left(1-\eta\left(x+e^{(2)}\right)\right) $;
	\item[2)]  $\eta^{(-1,-1)}_{\varepsilon N}(x) \eta\left(x+e^{(1)}+e^{(2)}\right)\left(1- \eta\left(x+e^{(2)}\right)\right)  $;
	\item[3)] $ \eta^{(-1,-1)}_{\varepsilon N}(x)\left(1- \eta^{(1,-1)}_{\varepsilon N}(x) \right)\left(1- \eta\left(x+e^{(2)}\right)\right) $;
	\item[4)] $ \eta^{(-1,-1)}_{\varepsilon N}(x)\left(1- \eta^{(1,-1)}_{\varepsilon N}(x) \right)\eta^{(1,1)}_{\varepsilon N}(x) $.
\end{itemize}
Analogous steps should be done also for $ g_2(\eta) $. We omit details.

\section{Generalized gradient models, weakly asymmetric models and Einstein relation}\label{norigor}

We give a short outline of the form of the scaling limits in several conditions. We give no proofs and our aim here is just to give a general overview.

\subsection{Scaling limits of generalized gradient models}
The first case we consider is a diffusive generalized gradient model, i.e. the instantaneous current is like in \eqref{alvar}, and having stationary grandcanonical measures $ \nu_\rho $ parameterized by the density $ \rho $. According to the general scheme of Section \ref{SL}, we have that $\mathcal J^N_t(G)$
is equal to
\begin{equation}\label{sera}
\frac{N^2}{2N^d}\int_0^tds\,\sum_{(x,y)\in E_N}G_N(x,y)j_{\eta_s}(x,y) \,,
\end{equation}
up to martingales terms negligible in the scaling limit.
Here $G$ is a $ C^\infty(\Lambda) $-vector field and $G_N$ its discretization given in \eqref{discvec}. The factor $N^2$ is due as usual to the diffusive rescaling of time. After some discrete integration by parts, we have that \eqref{sera} becomes
\begin{equation}\label{berrettini}
\frac{N^2}{N^d}\int_0^tds\,\sum_{x\in \Lambda_N}\sum_{i=1}^d\sum_{j=1}^d \tau_x h_{i,j}(\eta)\left(G_N(x-e^{(j)},x-e^{(j)}+e^{(i)})-G_N(x, x+e^{(i)})\right).
\end{equation}
We have that up to uniformly infinitesimal terms
$$
N^2 \left(G_N(x-e^{(j)},x-e^{(j)}+e^{(i)})-G_N(x, x+e^{(i)})\right)
$$
coincides with $-\partial_{x_j}G_i$.
We define $H^{i,j}(\rho)=- E_{\nu_\rho}\left(h_{i,j}\right)$ where we recall that $\nu_\rho$ is the grandcanonical invariant measure parameterized by the density $\rho$. By a replacement lemma we deduce that \eqref{berrettini} converges to
\begin{equation}
\int_0^t\,ds \int_\Lambda dx \sum_{i,j}H^{i,j}(\rho(t,x))\partial_{x_j}G_i(x)\,,
\end{equation}
where $\rho(x,t)$ is the solution of the hydrodynamic equation. This means that the typical current is
\begin{equation*}
J_i(\rho)=-\sum_{j=1}^d \partial_{x_j} \left(H^{i,j}(\rho)\right)\\
=-\sum_{j=1}^d\left(H_{i,j}\right)'(\rho)\partial_{x_j}\rho\,.
\end{equation*}
Recalling \eqref{tygr}, we have the non necessarily symmetric diffusion matrix
\begin{equation}\label{eq:nonsym}
\mathcal D_{i,j}(\rho)=\left(H_{i,j}\right)'(\rho)\,.
\end{equation}
The hydrodynamic equation is again the conservation law $\partial_t\rho+\nabla\cdot J(\rho)=0$.

\subsection{Weakly asymmetric models and Einstein relation}

We consider here the basic model \eqref{ratej} in presence of a weak external field. More precisely let $H=(H_1(x),H_2(x))$ be a $C^1$ vector field on $\Lambda$ and let $H_N$ be its discretized version given from \eqref{discvec}. For simplicity we consider a time independent vector field but all could be repeated in the case of a time dependent one.
We consider transition rates perturbed by the presence of the external field and defined by
\begin{equation}\label{ratesG}
c_{x,y}^H(\eta):=c_{x,y}(\eta)e^{H_N(x,y)}\,.
\end{equation}

Let us introduce the density of free energy
\begin{equation}\label{effe}
f(x)=x\log (x)+(1-x)\log(1-x)\,, \qquad x\in (0,1)\,,
\end{equation}
that coincides, up to a linear term, with the large deviations rate functional  for the stationary measure $ \nu_\rho $, that in this case is a product Bernoulli measure.

For the model with rates perturbed like in \eqref{ratesG} we have that the hydrodynamic equation is
\begin{equation}\label{caud2}
\partial_t\rho_t=\nabla \cdot\left(\nabla\rho_t-2\sigma(\rho)H\right)\,,
\end{equation}
and the corresponding typical current is
\begin{equation}\label{eq:currE}
J^H(\rho)=-\nabla\rho-A(\rho)\nabla\rho+2\sigma(\rho)H\,,
\end{equation}
where the mobility matrix is given by
\begin{equation}\label{sigma}
\sigma(\rho)=\mathbb I \rho(1-\rho)\,.
\end{equation}
We have therefore the validity of the Einstein relation
\begin{equation}\label{ein}
D(\rho)=\sigma(\rho)f''(\rho)\,,
\end{equation}
that involves only the symmetric part $D$ of the diffusion matrix.
	
An outline of the argument that gives \eqref{eq:currE}, \eqref{sigma} and \eqref{ein} is the following. We consider just the scaling of the current. By definition \eqref{discvec} we have that $H_N=O(1/N)$ and by a Taylor expansion we have that the instantaneous current $j_\eta^H$ for the perturbed model can be written up to uniformly infinitesimal terms as
	\begin{equation}\label{istcurG}
	j^H_\eta(x,y)=j_\eta(x,y)+\left[c_{x,y}(\eta)+c_{y,x}(\eta)\right]H_N(x,y)\,.
	\end{equation}
	The second term on the right-hand side of \eqref{istcurG} is
	\begin{equation}\label{sviluppo}
	\Big[(\eta(x)-\eta(y))^2+(\eta(x)-\eta(y))\Big(\tau_{f^+(x,y)}g(\eta)-\tau_{f^-(x,y)}g(\eta)\Big)\Big]H_N(x,y)\,.
	\end{equation}
	We obtain therefore \eqref{eq:currE} by a suitable replacement lemma and based on the following elementary computations. Recalling that $\nu_\rho$ in this case is a product Bernoulli measure, we have
	\begin{equation}\label{eunou}
	 E_{\nu_\rho}\left[(\eta(x)-\eta(y))^2\right]=2\rho(1-\rho)\,,
	\end{equation}
	while instead
	\begin{equation}\label{edueu}
	 E_{\nu_\rho}\left[(\eta(x)-\eta(y))\Big(\tau_{f^+(x,y)}g(\eta)-\tau_{f^-(x,y)}g(\eta)\Big)\right]=0\,.
	\end{equation}

	\section*{Acknowledgements}
 P.G. thanks  FCT/Portugal for support through the project  UID/MAT/04459/2013. This project has received funding from the European Research Council (ERC) under  the European Union's Horizon 2020 research and innovative programme (grant agreement   No 715734).  D.G. thanks L. Bertini and C. Landim for several discussions on the topological setting of section \ref{sec:currentop}, introduced in \cite{BLG}.

\end{document}